\documentclass[reqno,oneside,11pt]{amsart}

\usepackage[a4paper, hmargin={2.8cm, 2.8cm}, vmargin={2.5cm, 2.5cm}]{geometry}  
\usepackage[danish,english]{babel} 
\usepackage[latin1]{inputenc} 
\usepackage[T1]{fontenc} 
\usepackage{lmodern}
\usepackage{amsmath,amsthm,amssymb,amsfonts,mathrsfs,latexsym} 
\usepackage{graphicx} 
\usepackage{verbatim} 
\usepackage[all]{xy} 
\usepackage[pagebackref, colorlinks, linkcolor=red, citecolor=blue, urlcolor=blue, hypertexnames=true]{hyperref}  
\usepackage{multirow}
\usepackage{color}
\usepackage{tikz}
\usetikzlibrary{matrix,arrows}

\makeatletter
\newtheorem*{rep@theorem}{\rep@title}
\newcommand{\newreptheorem}[2]{%
\newenvironment{rep#1}[1]{%
 \def\rep@title{#2 \ref{##1}}%
 \begin{rep@theorem}}%
 {\end{rep@theorem}}}
\makeatother

\newtheorem{thm}{Theorem}[section]
\newtheorem{thmx}{Theorem}
\newtheorem{corx}[thmx]{Corollary}
\newreptheorem{thm}{Theorem}
\newtheorem{lem}[thm]{Lemma}
\newtheorem{prop}[thm]{Proposition}
\newtheorem{cor}[thm]{Corollary}
\newtheorem*{thm*}{Theorem}
\newtheorem*{problem*}{Problem}

\newtheorem*{claim*}{Claim}
\theoremstyle{definition}
\newtheorem{defi}[thm]{Definition}
\newtheorem{exam}[thm]{Example}

\newtheorem{rem}[thm]{Remark}
\newtheorem*{conjecture}{Conjecture}

\newcommand{\norm}[1] {\| #1 \|}
\newcommand{\abs}[1]{\lvert#1\rvert}

\newcommand{\C}{\mathbb{C}}
\newcommand{\N}{\mathbb{N}}

\newcommand{\R}{\mathbb{R}}
\newcommand{\Z}{\mathbb{Z}}

\renewcommand{\epsilon}{\varepsilon}
\renewcommand{\phi}{\varphi}
\renewcommand{\tilde}{\widetilde}
\renewcommand{\hat}{\widehat}

\DeclareMathOperator{\id}{id}

\DeclareFontFamily{U}{mathx}{\hyphenchar\font45}
\DeclareFontShape{U}{mathx}{m}{n}{
      <5> <6> <7> <8> <9> <10>
      <10.95> <12> <14.4> <17.28> <20.74> <24.88>
      mathx10
      }{}
\DeclareSymbolFont{mathx}{U}{mathx}{m}{n}
\DeclareFontSubstitution{U}{mathx}{m}{n}
\DeclareMathAccent{\widecheck}{0}{mathx}{"71}
\DeclareMathAccent{\wideparen}{0}{mathx}{"75}

\numberwithin{equation}{section}

\begin{document}
\selectlanguage{english} 
\date{\today}

\title[Ideal structure and pure infiniteness of ample groupoid $C^*$-algebras]{\texorpdfstring{Ideal structure and pure infiniteness of ample groupoid $C^*$-algebras}{Ideal structure and pure infiniteness of ample groupoid $C^*$-algebras}}

\author{Christian B\"onicke$^{1}$}
\address{Mathematisches Institut der WWU M\"unster,
\newline Einsteinstrasse 62, 48149 M\"unster, Germany}
\email{c.boenicke@wwu.de}

\author{Kang Li$^{2}$}
\address{Mathematisches Institut der WWU M\"unster,
\newline Einsteinstrasse 62, 48149 M\"unster, Germany}
\email{lik@uni-muenster.de}

\thanks{{$^{1}$} Supported by Deutsche Forschungsgemeinschaft (SFB 878).}

\thanks{{$^{2}$} Supported by the Danish Council for Independent Research (DFF-5051-00037) and partially supported by the DFG (SFB 878).}


\begin{abstract}
In this paper, we study the ideal structure of reduced $C^*$-algebras $C^*_r(G)$ associated to étale groupoids $G$. In particular, we characterize when there is a one-to-one correspondence between the closed, two-sided ideals in $C_r^*(G)$ and the open invariant subsets of the unit space $G^{(0)}$ of $G$. As a consequence, we show that if $G$ is an inner exact, essentially principal, ample groupoid, then $C_r^*(G)$ is (strongly) purely infinite if and only if every non-zero projection in $C_0(G^{(0)})$ is properly infinite in $C_r^*(G)$. We also establish a sufficient condition on the ample groupoid $G$ that ensures pure infiniteness of $C_r^*(G)$ in terms of paradoxicality of compact open subsets of the unit space $G^{(0)}$.

Finally, we introduce the type semigroup for ample groupoids and also obtain a dichotomy result:
Let $G$ be an ample groupoid with compact unit space which is minimal and topologically principal. If the type semigroup is almost unperforated, then $C_r^*(G)$ is a simple $C^*$-algebra which is either stably finite or strongly purely infinite.

\end{abstract}

\maketitle

\section{Introduction}

The notion of pure infiniteness for simple $C^*$-algebras was introduced by Cuntz in \cite{MR604046} as a $C^*$-algebraic analogue of the type III von Neumann algebras. The full importance of pure infiniteness became apparent when Kirchberg and Phillips proved that two simple, separable, nuclear and purely infinite $C^*$-algebras are stably isomorphic if and only if they are $KK$-equivalent (see \cite{K94,P00}).

In \cite{KR00} Kirchberg and Rørdam initiated the study of non-simple purely infinite $C^*$-algebras. This notion was subsequently strengthened by the same authors in \cite{KR02} to a notion termed strong pure infiniteness, which allows for the following formulation of a classification result of Kirchberg in the non-simple setting, announced in \cite{MR1796912}:
Two separable nuclear strongly purely infinite $C^*$-algebras, whose primitive ideal spaces are homeomorphic to the same space $X$, are stably isomorphic if and only if they are $KK_X$-equivalent, where $KK_X$ is a version of $KK$-theory that respects the primitive ideal spaces.

These classification results inspired many authors to study when different constructions of $C^*$-algebras are purely infinite. This led to a wealth of examples of (strongly) purely infinite $C^*$-algebras arising for example from dynamical systems (see \cite{MR2974211,MR3182957,RS12,MR1420560,MR1774856,PSS17}) or more generally from Fell bundles (see \cite{MR3543802}), and combinatorial objects like graphs (see for example \cite{MR1989499}).
Many of these examples can be viewed as étale groupoid $C^*$-algebras and hence the setting of groupoids serves as a convenient common framework.

Accordingly, this paper will be concerned with purely infinite $C^*$-algebras associated to étale groupoids.
One step in this direction has already been done in \cite{BCS15}, where the authors focus on the simple setting.
We extend the results of \cite{BCS15} in two directions: Firstly we drop the minimality condition, which leads us to the realm of non-simple $C^*$-algebras. And secondly, we give a sufficient condition on the groupoid for pure infiniteness of its reduced groupoid $C^*$-algebra.  
For the first part we wish to employ \cite[Proposition~4.7]{KR00}, which says that a $C^*$-algebra is purely infinite provided that every non-zero hereditary sub-$C^*$-algebra in every quotient contains an infinite projection. In fact, the latter property characterizes pure infiniteness of $C^*$-algebras with the ideal property (IP) (see \cite[Proposition~2.11]{PR07}). Thus, to achieve our goal, we first investigate the ideal structure of reduced groupoid $C^*$-algebras. Although many results on this topic are available for (partial) crossed products (see \cite{Sierakowski10,MR3182957}) and also some partial results for groupoids (see \cite{MR1191252,MR3606454}), a systematic treatment in the spirit of \cite{Sierakowski10} still seems to be missing.
We fill that gap by proving the following result:
\begin{thmx}[see Theorem~\ref{Thm:IdealStructure}]\label{ThmA}
Let $G$ be a locally compact Hausdorff étale groupoid. Then the following are equivalent:
\begin{itemize}
\item[(1)] $G$ is inner exact and satisfies the residual intersection property.
\item[(2)] The canonical map $I\mapsto U_I$, where

$$U_I=\bigcup\limits_{f\in I\cap C_0(G^{(0)})} f^{-1}(\C\setminus \lbrace0\rbrace),$$
gives rise to a one-to-one correspondence between the closed two-sided ideals of $C_r^*(G)$ and the open invariant subsets of $G^{(0)}$.
\end{itemize}
\end{thmx}
Note that the residual intersection property is satisfied whenever $G$ is essentially principal (the latter notion was introduced by Renault under the term essential freeness in \cite{MR1191252}).
As a consequence of Theorem \ref{ThmA}, we can identify the space of prime ideals of $C_r^*(G)$ with the quasi-orbit space of $G$ (see Theorem~\ref{Theorem:Prime Ideals}). Furthermore, Theorem \ref{ThmA} implies that an inner exact \textit{ample} groupoid satisfying the residual intersection property has the ideal property (IP) (see Corollary~\ref{Cor:IdealSeparationImplies(IP)}). This in turn allows us to deduce that pure infiniteness and strong pure infiniteness are the same for the $C^*$-algebras of the groupoids in question. In view of this observation, it seems natural to focus on the class of ample groupoids. This is an important class of examples, since every Kirchberg algebra satisfying the UCT is Morita equivalent to a $C^*$-algebra associated to a Hausdorff, ample groupoid (see \cite{MR2329002}).

With the description of the ideal structure at hand, we are able to prove an analogue of \cite[Theorem~4.1]{BCS15} in the non-minimal setting:
\begin{thmx}[see Theorem~\ref{Thm:CharacterizationPurelyInfinite}] \label{ThmB}
Suppose $G$ is a locally compact Hausdorff ample groupoid, such that $G$ is inner exact and essentially principal. Then the following are equivalent:
\begin{enumerate}
	\item $C_r^*(G)$ is (strongly) purely infinite.
	\item Every non-zero projection in $C_0(G^{(0)})$ is properly infinite in $C_r^*(G)$.
	\item Every non-zero projection in $C_0(D)$ is infinite in $C_r^*(G_D)$ for every closed $G$-invariant subset $D\subseteq G^{(0)}$.
\end{enumerate}
\end{thmx}

Our theorem reduces the question whether $C_r^*(G)$ is purely infinite (i.e. all non-zero positive elements are properly infinite) to whether all non-zero \emph{projections} coming from $C_0(G^{(0)})$ are properly infinite. Such projections will just be characteristic functions supported in compact open subsets of $G^{(0)}$. This gives hope to find sufficient conditions on the compact open subsets of $G^{(0)}$ to verify pure infiniteness of $C_r^*(G)$.
Inspired by the work of Rørdam and Sierakowski in \cite{RS12} and Kerr and Nowak in \cite{MR2974211}, we introduce a notion of paradoxical decomposition for subsets of the unit space and show how paradoxical decompositions yield (properly) infinite projections in $C_r^*(G)$.
As an immediate consequence, we obtain the following application to nuclear dimension of groupoid $C^*$-algebras:
\begin{corx}[see Corollary~\ref{Cor: Every set is paradoxical implies purely infinite}]
	Let $G$ be a locally compact Hausdorff ample groupoid, which is essentially principal and inner exact. Let $\mathcal{B}$ be a basis for the topology of $G^{(0)}$ consisting of compact open sets. Suppose each $A\in \mathcal{B}$ is paradoxical. Then $C_r^*(G)$ is (strongly) purely infinite.
\end{corx}

We also introduce an analogue of the type semigroup, which allows us to phrase paradoxicality in algebraic terms.
Very recently, Rainone showed in \cite[Theorem~5.6]{MR3669119} that for a large class of simple unital crossed product $C^*$-algebras there is a dichotomy: the $C^*$-algebra is either purely infinite or it admits a tracial state. Using our methods we can easily extend his result to the realm of groupoids:

\begin{thmx}[see Theorem~\ref{purely inf}] 
Let $G$ be a minimal and topologically principal ample groupoid with compact unit space. Assume moreover that the type semigroup is almost unperforated. Then the following are equivalent:
\begin{enumerate}
	\item The $C^*$-algebra $C_r^*(G)$ is (strongly) purely infinite.
	\item The $C^*$-algebra $C_r^*(G)$ admits no non-trivial lower-semicontinuous $2$-quasitraces.
	\item The $C^*$-algebra $C_r^*(G)$ is not stably finite.
\end{enumerate}
\end{thmx}

Finally, the following result shows the usefulness of paradoxical decompositions, which allow us to characterize the stable finiteness of $C_r^*(G)$ (compare \cite[Theorem~4.9]{MR3669119}):

\begin{thmx}[see Theorem~\ref{stably finite}]
	Let $G$ be a minimal ample groupoid with compact unit space. Then the following are equivalent:
	\begin{enumerate}
		\item The $C^*$-algebra $C_r^*(G)$ admits a faithful tracial state.
		\item The $C^*$-algebra $C_r^*(G)$ is stably finite.
		\item Every clopen subset of $G^{(0)}$ is completely non-paradoxical.
	\end{enumerate}
\end{thmx}

\section{Preliminaries}
In this section we recall some definitions and basic facts of the theory of étale groupoids and their reduced $C^*$-algebras. We also include detailed proofs for facts that we use extensively throughout the paper.
\subsection{Groupoids}
A \textit{groupoid} is a set $G$ together with a distinguished subset $G^{(2)}\subseteq G\times G$, called the set of \textit{composable pairs}, a product map $G^{(2)}\rightarrow G$ denoted by $(g,h)\mapsto gh$, and an inverse map $G\rightarrow G$, written $g\mapsto g^{-1}$, such that:
\begin{enumerate}
	\item If $(g_1,g_2),(g_2,g_3)\in G^{(2)}$, then so are $(g_1g_2,g_3)$ and $(g_1,g_2g_3)$ and their products coincide, meaning $(g_1g_2)g_3=g_1(g_2g_3)$;
	\item for all $g\in G$ we have $(g,g^{-1})\in G^{(2)}$; and
	\item for any $(g,h)\in G^{(2)}$ we have $g^{-1}(gh)=h$ and $(gh)h^{-1}=g$.
\end{enumerate}
Every groupoid comes with a subset
$$G^{(0)}=\lbrace gg^{-1}\mid g\in G\rbrace=\lbrace g^{-1}g\mid g\in G\rbrace$$
called the set of \textit{units} of $G$, and two maps $r,d:G\rightarrow G^{(0)}$ given by
$r(g)=gg^{-1}$ and $d(g)=g^{-1}g$ called \textit{range} and \textit{domain} maps respectively.
A subgroupoid of $G$ is a subset $H\subseteq G$ which is closed under the product and inversion meaning that $gh\in H$ for all $(g,h)\in G^{(2)}\cap H\times H$ and $g^{-1}\in H$ for all $g\in H$.\\
When $G$ is endowed with a locally compact Hausdorff topology under which the product and inversion maps are continuous, $G$ is called a locally compact groupoid. A \textit{bisection} is a subset $S\subseteq G$ such that the restrictions of the range and domain maps to $S$ are local homeomorphisms onto open subsets of $G$. We will denote the set of all open bisections by $G^{op}$. A locally compact, Hausdorff groupoid is called \textit{étale} if there is a basis for the topology of $G$ consisting of open bisections. It follows that $G^{(0)}$ is open in $G$. Recall that it is also closed, since $G$ is assumed to be Hausdorff. A topological groupoid is called \emph{ample} if it has a basis of compact open bisections. We will write $G^a$ for the subset of $G^{op}$ consisting of all compact open bisections. If $G$ is a locally compact, Hausdorff and étale groupoid, then $G$ is ample if and only if $G^{(0)}$ is totally disconnected (see \cite[Proposition 4.1]{E10}).

For a subset $D\subseteq G^{(0)}$ write
$$G_D:=\lbrace g\in G\mid d(g)\in D\rbrace,\ G^D:=\lbrace g\in G\mid r(g)\in D\rbrace,\ \text{and } G_D^D:=G_D\cap G^D.$$
If $D=\lbrace u\rbrace$ consists of a single point $u\in G^{(0)}$ we will omit the braces in our notation and write $G_u:=G_D$, $G^u:=G^D$ and $G_u^u:=G_D^D$. Note that $G_u^u$ is a group, often called the \textit{isotropy} at $u$. If all the isotropy groups are trivial, i.e. $G_u^u=\lbrace u\rbrace$ for all $u\in G^{(0)}$, we say that $G$ is a \textit{principal} groupoid. There is also a weaker topological version of principality: We say that $G$ is \textit{topologically principal} if the set of all units with trivial isotropy $\lbrace u\in G^{(0)}\mid G_u^u=\lbrace u\rbrace\rbrace$ is dense in $G^{(0)}$.

A subset $D\subseteq G^{(0)}$ is called \textit{invariant}, if $d(g)\in D$ implies $r(g)\in D$ for all $g\in G$. This is equivalent to saying that $G_D=G^D$ and hence $G_D$ is a subgroupoid of $G$ with unit space $D$. Furthermore note that $D$ is invariant if and only if $G^{(0)}\setminus D$ is invariant. If $G$ is étale and $D$ is a closed invariant subset of $G^{(0)}$, then $G_D$ and $G_{G^{(0)}\setminus D}$ will be étale groupoids in their own right with base spaces $D$ and $G^{(0)}\setminus D$ respectively.

We say that $G$ is \textit{minimal}, if there are no proper, non-trivial closed invariant subsets in $G^{(0)}$.
Moreover $G$ is called \textit{essentially principal} if the closed subgroupoid $G_D$ is topologically principal for every closed invariant subset $D\subseteq G^{(0)}$.
\subsection{Groupoid $C^*$-algebras}
Let $G$ be an étale groupoid. Consider the complex vector space $C_c(G)$ of compactly supported complex-valued continuous functions on $G$. Since each $G^u$ is discrete we can define a convolution product on $C_c(G)$ by
$$f_1\ast f_2(g)=\sum\limits_{h\in G^{r(g)}}f_1(h)f_2(h^{-1}g)$$
and an involution by the formula
$$f^*(g)=\overline{f(g^{-1})}.$$
With these operations $C_c(G)$ becomes a $*$-algebra. For each $u\in G^{(0)}$ consider the regular representation $\pi_u:C_c(G)\rightarrow B(\ell^2(G_u))$ given by 
$$\pi_u(f)\xi=f\ast\xi$$
The reduced $C^*$-norm on $C_c(G)$ is defined by
$$\norm{f}_r:=\sup\limits_{u\in G^{(0)}}\norm{\pi_u(f)}$$
and we let $C_r^*(G)$ be the completion of $C_c(G)$  with respect to the reduced norm $\norm{\cdot}_r$.

Let us now recall some well-known facts, that we will use extensively throughout the paper. If $G$ is an étale groupoid, there is a canonical map $\iota_0:C_c(G^{(0)})\rightarrow C_c(G)$ given by extending a function by zero on $G\setminus G^{(0)}$.
For $f_1,f_2\in C_c(G^{(0)})$ we compute
$$f_1\ast f_2(g)=\sum\limits_{h\in G^{r(g)}}f_1(h)f_2(h^{-1}g)=f_1(r(g))f_2(g).$$
The latter product is zero unless $g\in G^{(0)}$ in which case $g=r(g)$. Hence $\iota_0$ is a $*$-homomorphism.
Let us show that $\iota_0$ is isometric:
For $f\in C_c(G^{(0)})$ and $\xi\in \ell^2(G_u)$ we have
$(\pi_u(f)\xi)(h)=f(r(h))\xi(h)$ and hence $$\norm{\pi_u(f)\xi}^2=\sum\limits_{h\in G_u} \abs{(\pi_u(f)\xi)(h)}^2 = \sum\limits_{h\in G_u} \abs{f(r(h))\xi(h)}^2\leq \norm{f}_\infty^2 \norm{\xi}^2.$$ 
It follows that $\norm{\pi_u(f)}\leq \norm{f}_\infty$ for all $u\in G^{(0)}$ and hence $\norm{f}_r\leq \norm{f}_\infty$.
For the converse inequality we use that for every $u\in G^{(0)}$ we have $\pi_u(f)\delta_u=f(u)\delta_u$ to compute $\abs{f(u)}=\norm{\pi_u(f)\delta_u}\leq \norm{\pi_u(f)}\leq \norm{f}_r$. Hence we have $\norm{f}_\infty=\norm{f}_r$ as desired.
Consequently, $\iota_0$ extends to an embedding $\iota:C_0(G^{(0)})\hookrightarrow C_r^*(G)$.

In the other direction we consider the restriction map $E_0:C_c(G)\rightarrow C_c(G^{(0)})$. It is clearly linear and restricts to the identity on $C_c(G^{(0)})$. Note that for any $u\in G^{(0)}$ and $f\in C_c(G)$ we have the identity
\begin{equation}\label{FormulaRegularRep}
\pi_u(f)\delta_u=\sum\limits_{g\in G_u}f(g)\delta_g
\end{equation}
in $\ell^2(G_u)$. And hence for any $f\in C_c(G)$ we have
\begin{equation}\label{FormulaCompactlySupportedFunction}
f(g)=\langle \pi_{d(g)}(f)\delta_{d(g)},\delta_g\rangle.
\end{equation}
Applying the Cauchy-Schwartz inequality we get the estimate
$$\abs{f(g)}\leq \norm{\pi_{d(g)}(f)}\leq \norm{f}_r$$
for all $g\in G^{(0)}$. Consequently, $\norm{E_0(f)}\leq \norm{f}_r$ for all $f\in C_c(G)$. Thus we have $\norm{E_0}\leq 1$. It is not hard to see that actually $\norm{E_0}=1$ by picking a function $f\in C_c(G^{(0)})$ with $\norm{f}_\infty=1$ and using that $\iota$ is isometric.

Consequently, $E_0$ extends to a linear map $E:C_r^*(G)\rightarrow C_0(G^{(0)})$ of norm $1$, which restricts to the identity on $C_0(G^{(0)})$. One easily computes
$$E(f^*\ast f)(u)=\sum\limits_{g\in G^u}\abs{f(g^{-1})}^2$$
Hence $E(f^*\ast f)\geq 0$ for all $f\in C_c(G)$ and by continuity we obtain that $E$ is positive. Thus, $E$ is a conditional expectation.
Note that by (\ref{FormulaCompactlySupportedFunction}) we have
$E(f)(u)=\langle \pi_u(f)\delta_u,\delta_u\rangle$. By continuity of $E$ we get the formula
\begin{equation}\label{FormulaConditionalExpectation}
E(x)(u)=\langle \pi_u(x)\delta_u,\delta_u\rangle
\end{equation}
for all $x\in C_r^*(G)$ and $u\in G^{(0)}$.

We claim that $E$ is faithful. So let $x\in C_r^*(G)$ with $E(x^*x)=0$. To see that $x=0$ it is enough to show that $\pi_u(x)=0$ for all $u\in G^{(0)}$.
So fix $u\in G^{(0)}$ and $g\in G_u$. Let $\varphi\in C_c(G)$ be supported in a bisection with $\varphi(g)=1$ and $0\leq \varphi\leq 1$.
Then we have $\pi_u(\varphi)\delta_u=\delta_g$ in $\ell^2(G_u)$ by equation \ref{FormulaRegularRep}.
Now for any $f\in C_c(G)$ one has $\abs{E(\varphi^*\ast f\ast \varphi)(u)}=\abs{\varphi(g^{-1})f(d(g))\varphi(g^{-1})}\leq \abs{f(d(g))}\leq \norm{E(f)}_\infty$. By continuity we get $\abs{E(\varphi^*\ast y\ast \varphi)(u)}\leq\norm{E(y)}_\infty$ for all $y\in C_r^*(G)$. In particular, we have 
$$0=\norm{E(x^*x)}_\infty \geq \abs{E(\varphi^*x^*x\varphi)(u)}=\abs{\langle \pi_u(\varphi^*x^*x\varphi)\delta_u,\delta_u\rangle}=\norm{\pi_u(x)\delta_g}^2$$
and hence conclude $\pi_u(x)=0$.

\section{Separating Ideals}
In this section we will describe our results on the ideal structure of reduced groupoid $C^*$-algebras.
Let us first recall some basic facts from the theory of ideals in $C^*$-algebras:
For a $C^*$-algebra $A$, let $\mathcal{I}(A)$ be the lattice of closed, two-sided ideals in $A$. Equip $\mathcal{I}(A)$ with 
the \textit{Fell topology} a sub-base for which is given by sets of the form
$$\mathcal{O}_I:=\lbrace J\in \mathcal{I}(A)\mid I\not\subseteq J\rbrace, \ I\in \mathcal{I}(A).$$
If $A$ and $B$ are two $C^*$-algebras and $\varphi:\mathcal{I}(A)\rightarrow\mathcal{I}(B)$ is a map preserving the inclusion of ideals, then $\varphi$ is continuous.

Throughout this section $G$ will be an étale (locally compact Hausdorff) groupoid. We will consider the lattice $\mathcal{O}(G^{(0)})$ of all open invariant subsets of $G^{(0)}$. If we identify $\mathcal{I}(C_0(G^{(0)}))$ with the set of open subsets of $G^{(0)}$, we obtain a natural topology on $\mathcal{O}(G^{(0)})$ by restricting the Fell topology from $\mathcal{I}(C_0(G^{(0)}))$.

Let $I\in \mathcal{I}(C_r^*(G))$ be an ideal. Then $I\cap C_0(G^{(0)})$ is an ideal in $C_0(G^{(0)})$. Hence it corresponds to the open subset
$$U_I:=\bigcup\limits_{f\in I\cap C_0(G^{(0)})} f^{-1}(\C\setminus \lbrace0\rbrace).$$
\begin{lem}\label{Lemma:IdealsGiveInvariantSubsets} Let $G$ be an étale groupoid and $I,J\in \mathcal{I}(C_r^*(G))$ be two ideals. Then the following hold:
\begin{enumerate}
\item The open subset $U_I\subseteq G^{(0)}$ is invariant.
\item If $I\subseteq J$ then $U_I\subseteq U_J$.
\item We have $U_I\cap U_J=U_{I\cap J}$ and $U_I\cup U_J\subseteq U_{I+J}$.
\end{enumerate}	
\end{lem}
\begin{proof}
	To show $(1)$ let $g\in G$ with $d(g)\in U_I$. Then we can find $f\in I\cap C_0(G^{(0)})$ such that $f(d(g))\neq 0$. Choose an open bisection $V$ of $G$ with $g\in V$ and find a compactly supported function $\phi\in C_c(V)\subseteq C_r^*(G)$ such that $\phi(g)=1$. Then $\phi\ast f\ast \phi^*\in I\cap C_0(G^{(0)})$ and since $\phi$ is supported in a bisection and $\phi(g)=1$ we have $$\phi\ast f\ast \phi^*(r(g))=f(d(g))\neq 0.$$
	We conclude that $r(g)\in U_I$. Since $(2)$ is immediate, it remains to prove $(3)$. Note that by $(2)$ we have $U_{I\cap J}\subseteq U_I\cap U_J$. If conversely $u\in U_I\cap U_J$, then there exist functions $f_1\in I\cap C_0(G^{(0)})$ and $f_2\in J\cap C_0(G^{(0)})$ such that $f_1(u)\neq 0\neq f_2(u)$. Then $f_1\ast f_2\in I\cap J\cap C_0(G^{(0)})$ does the job. 
	The last assertion follows immediately from $(2)$.
\end{proof}

From Lemma~\ref{Lemma:IdealsGiveInvariantSubsets} we obtain a well-defined continuous map 
$$\Theta:\mathcal{I}(C_r^*(G))\rightarrow \mathcal{O}(G^{(0)})$$
given by $\Theta(I):=U_I$.
Conversely there is an obvious map $$\Xi:\mathcal{O}(G^{(0)})\rightarrow\mathcal{I}(C_r^*(G))$$ sending an open invariant subset to the ideal $Ideal_{C_r^*(G)}[C_0(U)]$ of $C_r^*(G)$ generated by $C_0(U)$. Since $\Xi$ respects inclusions, it is continuous as well.
\begin{lem}\label{Lem:Surjectivity}
	Let $U\subseteq G^{(0)}$ be an open invariant subset. Then we have
	$C_0(U)=C_0(G^{(0)})\cap Ideal_{C_r^*(G)}[C_0(U)]$, i.e. $\Theta\circ \Xi=id_{\mathcal{O}(G^{(0)})}$.
\end{lem}
\begin{proof}
It is clear that $C_0(U)\subseteq C_0(G^{(0)})\cap Ideal_{C_r^*(G)}[C_0(U)]$. For the converse pick $f\in C_0(G^{(0)})\cap Ideal_{C_r^*(G)}[C_0(U)]$. It is enough to show that if $x\in G^{(0)}$ with $f(x)\neq 0$, then $x\in U$. Since $f$ is contained in the ideal generated by $C_0(U)$ we can approximate it by finite sums of functions of the form $\phi\ast \tilde{f}\ast\psi$ where $\phi,\psi\in C_c(G)$ and $\tilde{f}\in C_c(U)$. In particular there must exist such functions with the property that $\varphi\ast\tilde{f}\ast\psi(x)\neq 0$.
Hence there are elements $g,h\in G^x$ such that
$$0\neq \varphi\ast\tilde{f}\ast\psi(x)=\varphi(h)\tilde{f}(h^{-1}g)\psi(g^{-1})$$
It follows that $h^{-1}g$ must be an element of $U$ since $supp(\tilde{f})$ is contained in $U$. But then $h\in G$ with $d(h)=d(g^{-1}h)=g^{-1}h\in U$ and by invariance of $U$ we conclude that $x=r(h)\in U$ as desired.
\end{proof}

It is also worthwhile to note the following lemma, which we will use freely.
\begin{lem}\label{Lemma:RestrictedGroupoidIsIdeal}
For an open invariant subset $U\subseteq G^{(0)}$ we can identify the ideal of $C_r^*(G)$ generated by $C_0(U)$ with the reduced groupoid $C^*$-algebra $C_r^*(G_U)$.
\end{lem}
\begin{proof}
	As $U$ is open in $G$, we can identify $C_r^*(G_U)$ with the closure of $C_c(G_U)$ in $C_r^*(G)$.
	Let $f\in C_c(G)$ and $g\in C_c(G_U)$. Then $supp(f\ast g)\subseteq supp(f)supp(g)\subseteq G_U$. Since $C_r^*(G_U)$ is closed in $C_r^*(G)$ and convolution is continuous, we obtain that $C_r^*(G_U)$ is an ideal. Consequently we have $Ideal_{C_r^*(G)}[C_0(U)]\subseteq C_r^*(G_U)$. If conversely $x\in C_r^*(G_U)$ is a non-zero positive element then we can use the fact that $C_0(U)$ contains an approximate unit $(f_\lambda)_\lambda$ for $C_r^*(G_U)$ (see e.g. \cite[Lemma~2.1 (5)]{BCS15}) to write $x=\lim_\lambda x^{\frac{1}{2}}f_\lambda x^{\frac{1}{2}}$.
	But $x^{\frac{1}{2}}f_\lambda x^{\frac{1}{2}}$ is obviously in the ideal of $C_r^*(G)$ generated by $C_0(U)$.
\end{proof}

In the following we want to characterize when the ideal lattice of $C_r^*(G)$ is completely described by the open invariant subsets of $G^{(0)}$.
We say that $C_0(G^{(0)})$ \textit{separates the ideals of} $C_r^*(G)$ if the canonical map $\Theta:\mathcal{I}(C_r^*(G))\rightarrow \mathcal{O}(G^{(0)})$ is injective. Note that by Lemma \ref{Lem:Surjectivity} it is always surjective. The following elementary lemma is very helpful, as it allows us to phrase the question whether $C_0(G^{0})$ separates the ideals in terms of the canonical conditional expectation $E:C_r^*(G)\rightarrow C_0(G^{(0)})$:

\begin{lem}\label{Lemma:CharacterizationsOfIdealSeparation}
	Let $G$ be an étale groupoid and $E:C_r^*(G)\rightarrow C_0(G^{(0)})$ be the canonical conditional expectation. Then the following are equivalent:
	\begin{enumerate}
		\item $I=Ideal_{C_r^*(G)}[I\cap C_0(G^{(0)})]$ for every ideal $I\subseteq C_r^*(G)$.
		\item $C_0(G^{(0)})$ separates the ideals of $C_r^*(G)$.
		\item The map $\Xi:\mathcal{O}(G^{(0)})\rightarrow \mathcal{I}(C_r^*(G))$
		is surjective.
		\item $I=Ideal_{C_r^*(G)}[E(I)]$ for every ideal $I\subseteq C_r^*(G)$.
	\end{enumerate}
\end{lem}
\begin{proof}
	$(1)\Rightarrow(2)$: Suppose $I_1,I_2$ are ideals in $C_r^*(G)$ such that $I_1\cap C_0(G^{(0)})=I_2\cap C_0(G^{(0)})$. Then by $(1)$ we have
	$$I_1=Ideal_{C_r^*(G)}[I_1\cap C_0(G^{(0)})]=Ideal_{C_r^*(G)}[I_2\cap C_0(G^{(0)})]=I_2$$
	$(2)\Rightarrow(3)$: Since we have $\Theta\circ \Xi=\id_{\mathcal{O}(G^{(0)})}$ by Lemma \ref{Lem:Surjectivity}, injectivity of $\Theta$ implies surjectivity of $\Xi$.
	
	$(3)\Rightarrow (4)$: Property $(4)$ clearly holds for ideals of the form $C_r^*(G_U)$, where $U\subseteq G^{(0)}$ is open and $G$-invariant, since
	$$C_r^*(G_U)=Ideal_{C_r^*(G)}[C_0(U)]=Ideal_{C_r^*(G)}[E(C_r^*(G_U))].$$
	But $(3)$ says that every ideal of $C_r^*(G)$ is of this form.
	
	$(4)\Rightarrow(1)$: Let $I\subseteq C_r^*(G)$ be an ideal. Then $$I\cap C_0(G^{(0)})\subseteq E(I)\subseteq Ideal_{C_r^*(G)}[E(I)]\cap C_0(G^{(0)})=I\cap C_0(G^{(0)}).$$
	Thus $E(I)=I\cap C_0(G^{(0)})$ and using $(4)$ again we get $$I=Ideal_{C_r^*(G)}[E(I)]=Ideal_{C_r^*(G)}[I\cap  C_0(G^{(0)})].$$
\end{proof}
We shall need the following definition from \cite{AD16}:
\begin{defi}
	Let $G$ be a groupoid. $G$ is called \textit{inner exact} if the sequence
	$$0\rightarrow C_r^*(G_{G^{(0)}\setminus D})\rightarrow C_r^*(G)\rightarrow C_r^*(G_D)\rightarrow 0$$
	 is exact for every closed $G$-invariant subset $D\subseteq G^{(0)}$.
\end{defi}

\begin{lem}\label{Lemma:IdealInclusionAndInnerExactness}
	Let $G$ be an étale groupoid and $E:C_r^*(G)\rightarrow C_0(G^{(0)})$ be the canonical conditional expectation. Then the following are equivalent:
	\begin{enumerate}
		\item $G$ is inner exact.
		\item $x\in Ideal_{C_r^*(G)}[E(x)]$ for all $x\in C_r^*(G)_+$.
		\item $I\subseteq Ideal_{C_r^*(G)}[E(I)]$ for all ideals $I\subseteq C_r^*(G)$.
	
	\end{enumerate}
	\end{lem}
	\begin{proof}
	$(1)\Rightarrow (2)$: Let $x\in C_r^*(G)_+$. Let $I$ be the smallest $G$-invariant ideal of $C_0(G^{(0)})$ containing $E(x)$ and let $U$ be the corresponding $G$-invariant open subset with $I= C_0(U)$. Let $D=G^{(0)}\setminus U$. Then we obtain a commutative diagram with exact rows:\\
	\begin{center}
		\begin{tikzpicture}[description/.style={fill=white,inner sep=2pt}]
		\matrix (m) [matrix of math nodes, row sep=3em,
		column sep=2.5em, text height=1.5ex, text depth=0.25ex]
		{ 0 & C_r^*(G_{U})\ & C_r^*(G) & C_r^*(G_{D}) & 0 \\
			0 & C_0(U)\ & C_0(G^{(0)}) & C_0(D) & 0 \\
		};
		\path[->,font=\scriptsize]
		(m-1-1) edge node[auto] {$  $} (m-1-2)
		(m-1-2) edge node[auto] {$  $} (m-1-3)
		(m-1-3) edge node[auto] {$ \pi $} (m-1-4)
		(m-1-4) edge node[auto] {$  $} (m-1-5)
		(m-2-1) edge node[auto] {$  $} (m-2-2)
		(m-2-2) edge node[auto] {$  $} (m-2-3)
		(m-2-3) edge node[auto] {$ res_D $} (m-2-4)
		(m-2-4) edge node[auto] {$  $} (m-2-5)
		(m-1-2) edge node[auto] {$ E_U  $} (m-2-2)
		(m-1-3) edge node[auto] {$ E  $} (m-2-3)
		(m-1-4) edge node[auto] {$ E_D  $} (m-2-4)
		;
		\end{tikzpicture}
	\end{center}
	Since $E(x)\in I\cong C_0(U)$ by definition, we have that $E_D(\pi(x))=res_D(E(x))=0$. As $E_D$ is faithful we obtain $\pi(x)=0$, which implies $x\in C_r^*(G_U)=Ideal_{C_r^*(G)}[C_0(U)]=Ideal_{C_r^*(G)}[E(x)]$.\\
	$(2)\Rightarrow (3)$: Let $I$ be an ideal of $C_r^*(G)$. Then $I_+\subseteq\bigcup\limits_{x\in I_+} Ideal_{C_r^*(G)}[E(x)]\subseteq Ideal_{C_r^*(G)}[E(I)]$.
	Thus, we have $I\subseteq Ideal_{C_r^*(G)}[E(I)]$.\\
	$(3)\Rightarrow(1)$: Let $D\subseteq G^{(0)}$ be a closed $G$-invariant subset and $U=G^{(0)}\setminus D$. We only need to show that $ker(\pi)\subseteq C_r^*(G_U)$. Let $I:=ker(\pi)$. Then $res_D(E(I))=E_D(\pi(I))=0$ and by exactness of the lower row of the diagram $E(I)\subseteq C_0(U)$. Applying $(3)$ we obtain $ker(\pi)=I\subseteq Ideal_{C_r^*(G)}[E(I)]\subseteq Ideal_{C_r^*(G)}[C_0(U)]=C_r^*(G_U)$.
	\end{proof}
	
	\begin{lem}\label{inner exact and intersection}
		Let $G$ be an inner exact, étale groupoid. If $I\subseteq C_r^*(G)$ is an ideal, and $D\subseteq G^{(0)}$ is the closed invariant subset corresponding to $I\cap C_0(G^{(0)})$, then $\pi(I)\cap C_0(D)=\lbrace 0\rbrace$, where $\pi:C_r^*(G)\rightarrow C_r^*(G_D)$ is the quotient map.
	\end{lem}
	\begin{proof}
		Let $x\in I$ such that $\pi(x)\in C_0(D)$. Then $\pi(E(x)-x)=res_D(E(x))-\pi(x)=E_D(\pi(x))-E_D(\pi(x))=0$. By inner exactness we have $$E(x)-x\in C_r^*(G_U)=Ideal_{C_r^*(G)}[C_0(U)]\subseteq I,$$ where $U=G^{(0)}\setminus D$. Since $x\in I$ as well, we conclude $E(x)\in I\cap C_0(G^{(0)})=C_0(U)$. But then $0=res_D(E(x))=E_D(\pi(x))=\pi(x)$.
	\end{proof}
	
	\begin{defi}
		Let $G$ be an étale groupoid. We say that $G$ has the \emph{intersection property}, if for every non-zero ideal $I\subseteq C_r^*(G)$ we have $I\cap C_0(G^{(0)})\neq \lbrace 0\rbrace$. If $G_D$ has the intersection property for every closed, $G$-invariant subset $D\subseteq G^{(0)}$, we say that $G$ has the \emph{residual intersection property}.
	\end{defi}
	
	\begin{prop}\label{Prop:IdealInclusionAndResidualIntersection}
		Let $G$ be an étale groupoid. Then the following are equivalent:
		\begin{enumerate}
			\item For every ideal $I\subseteq C_r^*(G)$ we have $Ideal_{C_r^*(G)}[E(I)]\subseteq I$.
			\item For every $x\in C_r^*(G)_+$ we have $E(x)\in Ideal_{C_r^*(G)}[x]$.
			\item $G$ has the residual intersection property and for every ideal $I\subseteq C_r^*(G)$ we have $\pi(I)\cap C_0(D)=\lbrace 0\rbrace$, where $\pi$ and $D$ are as in Lemma~\ref{inner exact and intersection}.
		\end{enumerate}
	\end{prop}
	\begin{proof}
		We start by showing $(1)\Rightarrow (3)$. For the residual intersection property let $D\subseteq G^{(0)}$ be a closed $G$-invariant subset and $I\subseteq C_r^*(G_D)$ an ideal. Suppose that $I\cap C_0(D)= \lbrace 0\rbrace$. Consider the ideal $J:=\pi^{-1}(I)\subseteq C_r^*(G)$. Then by $(1)$ we have $E(J)\subseteq J\cap C_0(G^{(0)})$ and thus $E_D(I)=E_D(\pi(J))=res_D(E(J))\subseteq res_D(J\cap C_0(G^{(0)}))\subseteq I\cap C_0(D)=\lbrace 0\rbrace$. But then also $I=\lbrace 0\rbrace$.
		
		For the second part let $I\subseteq C_r^*(G)$ be any ideal and $x\in I$, such that $\pi(x)\in C_0(D)$. By $(1)$ we have $E(x)\in I$ and thus $0=res_D(E(x))=E_D(\pi(x))=\pi(x)$.
		
		To show $(3)\Rightarrow (2)$ let $x\in C_r^*(G)_+$ be given. Let $I:=Ideal_{C_r^*(G)}[x]$ and $D$ be the corresponding closed, $G$-invariant subset of $G^{(0)}$. Applying the second part of $(3)$, we obtain $\pi(I)\cap C_0(D)=\lbrace 0\rbrace$. Since $G$ has the residual intersection property this can only happen if the ideal $\pi(I)$ is zero. Thus $res_D(E(x))=E_D(\pi(x))=0$ and $E(x)\in ker(res_D)=I\cap C_0(G^{(0)})\subseteq I$.
		
		The last implication is $(2)\Rightarrow (1)$: Let $I\subseteq C_r^*(G)$ be an ideal. By $(2)$ we have $E(I_+)\subseteq I$. Since every element in $I$ is a linear combination of positive elements and $E$ is linear we get $E(I)\subseteq I$ and then also $Ideal_{C_r^*(G)}[E(I)]\subseteq I$.
	\end{proof}
		
 We arrive at the main result of this section:
\begin{thm}\label{Thm:IdealStructure}
	Let $G$ be an étale groupoid. Then the following are equivalent:
	\begin{enumerate}
		\item $\Theta:\mathcal{I}(C_r^*(G))\rightarrow \mathcal{O}(G^{(0)})$ is a bijection.
		\item $C_0(G^{(0)})$ separates the ideals of $C_r^*(G)$.
		\item $G$ is inner exact and satisfies the residual intersection property.
	\end{enumerate}
	Moreover, if the above properties are satisfied, then $\Theta$ is an isomorphism of lattices and hence a homeomorphism.
\end{thm}
\begin{proof}
	Everything follows by combining Proposition \ref{Prop:IdealInclusionAndResidualIntersection} and Lemmas \ref{Lemma:CharacterizationsOfIdealSeparation}, \ref{Lemma:IdealInclusionAndInnerExactness} and \ref{inner exact and intersection} except that $\Theta$ is an isomorphism of lattices. By Lemma \ref{Lemma:IdealsGiveInvariantSubsets} it only remains to check that $U_{I+J}\subseteq U_I\cup U_J$, whenever $I,J\in\mathcal{I}(C_r^*(G))$. It is enough to show that $(I+J)\cap C_0(G^{(0)})\subseteq (I\cap C_0(G^{(0)}))+(J+C_0(G^{(0)}))$. Suppose that $a\in I$ and $b\in J$ such that $a+b\in C_0(G^{(0)})$. Then, by Proposition~\ref{Prop:IdealInclusionAndResidualIntersection} we have $E(a)\in Ideal_{C_r^*(G)}[a]\subseteq I$. Thus $E(a)\in I\cap C_0(G^{(0)})$. Similarly we have $E(b)\in J\cap C_0(G^{(0)})$. Thus we obtain $a+b=E(a+b)=E(a)+E(b)\in (I\cap C_0(G^{(0)}))+(J\cap C_0(G^{(0)}))$ as desired.
\end{proof}

In practice the (residual) intersection property is of course not very satisfying as it is formulated on the level of the reduced groupoid $C^*$-algebra rather then the groupoid itself.
The following result gives at least a sufficient criterion in terms of $G$ for the intersection property to hold and shows that it is also necessary in the case that $C^*(G)=C_r^*(G)$ (see \cite{BCFS14} for an easy definition of the full $C^*$-algebra of an étale groupoid).
\begin{prop}{\cite[Proposition 5.5]{BCFS14}}\label{Prop:TopPrincipalAndNontrivialIntersections}
Let $G$ be an étale groupoid. If $G$ is topologically principal, then $G$ has the intersection property.
If $C^*(G)=C_r^*(G)$ and $G$ has the intersection property, then $G$ is topologically principal.
\end{prop}

As a consequence we obtain:
\begin{cor}\label{essprin+inner exact}
If $G$ is an essentially principal, inner exact, étale groupoid, then $C_0(G^{(0)})$ separates the ideals of $C_r^*(G)$. If $G$ is amenable, then $C_0(G^{(0)})$ separates the ideals of $C_r^*(G)$ if and only if $G$ is essentially principal.
\end{cor}
\begin{proof}
We only need to show, that $G$ has the residual intersection property. So let $D\subseteq G^{(0)}$ be a closed $G$-invariant subset and $I\subseteq C_r^*(G_D)$ a non-zero ideal. Since $G$ is essentially principal, $G_D$ is topologically principal and thus, by Proposition \ref{Prop:TopPrincipalAndNontrivialIntersections}, we obtain $I\cap C_0(D)\neq\lbrace 0\rbrace$.\\
If $G$ is amenable we have $C^*(G)=C_r^*(G)$ and a quick diagram chase gives that also $C^*(G_D)=C_r^*(G_D)$ for all closed $G$-invariant subsets $D\subseteq G^{(0)}$. Again the result follows directly from Proposition \ref{Prop:TopPrincipalAndNontrivialIntersections} and Theorem \ref{Thm:IdealStructure}.
\end{proof}
\begin{rem}
Recently, it was proved by Tikuisis, White and Winter that every faithful tracial state on a separable nuclear $C^*$-algebra in the UCT class is quasidiagonal \cite[Theorem~A]{MR3583354}. Hence, it follows from \cite[Corollary~4.9]{MR1191252} or our previous corollary that if $G$ is a second countable étale groupoid which is amenable and essentially principal, then every tracial state on $C_r^*(G)$ is quasidiagonal. Indeed, let $\tau$ be any tracial state on $C_r^*(G)$. If $\pi_\tau$ denotes the GNS representation of $\tau$, then $\tau$ vanishes on the two-sided closed ideal $I:=\ker \pi_\tau$. Thus, there is a faithful induced traical state $\dot{\tau}$ on $C_r^*(G)/I$ such that $\tau=\dot{\tau}\circ p$, where $p:C_r^*(G) \rightarrow C_r^*(G)/I$ is the quotient map. Clearly, it suffices to show that $\dot{\tau}$ is quasidiagonal. Since $C_0(G^{(0)})$ separates the ideals of $C_r^*(G)$, there is a closed $G$-invariant subset $D$ of $G^{(0)}$ such that $C_r^*(G)/I=C_r^*(G_D)$. As $G_D$ is second countable and amenable, it follows that $C_r^*(G_D)$ is a separable nuclear $C^*$-algebra in the UCT class (see \cite{MR1703305} for UCT). We complete our proof by \cite[Theorem~A]{MR3583354}.
\end{rem}

\begin{cor}\label{simple}
	Let $G$ be an étale groupoid such that there exists a unit $u\in G^{(0)}$ with $G_u^u=\lbrace u\rbrace$ (e.g. if $G$ is topologically principal). Then $C_r^*(G)$ is simple if and only if $G$ is minimal.
\end{cor}
\begin{proof}
Suppose that $G$ is minimal.
	By assumption $\lbrace u\in G^{(0)}\mid G_u^u=\lbrace u\rbrace\rbrace$ is non-empty and its closure is a closed invariant subset of the unit space. Using minimality we conclude, that it must be all of $G^{(0)}$. Hence $G$ is topologically principal. Since every minimal groupoid is automatically inner exact and in the minimal case topological principality coincides with essential principality, $C_r^*(G)$ is simple by Theorem \ref{Thm:IdealStructure}.
	If $C_r^*(G)$ is simple, then the ideal of $C_r^*(G)$ generated by a non-zero $f\in C_c(G^{(0)})\subseteq C_r^*(G)$ is of course all of $C_r^*(G)$ and hence Proposition 5.7 of \cite{BCFS14} implies that $G$ is minimal. 
\end{proof}
Recall that a $C^*$-algebra has the \textit{ideal property} (IP) if every closed two-sided ideal is generated (as an ideal) by its projections.
\begin{cor}\label{Cor:IdealSeparationImplies(IP)}
	Let $G$ be an ample groupoid. If $G$ is inner exact and has the residual intersection property, then $C_r^*(G)$ has the ideal property (IP).
\end{cor}
\begin{proof}
	Let $I\subseteq C_r^*(G)$ be an ideal. By our assumptions on $G$ and Theorem \ref{Thm:IdealStructure} we know that $I$ is generated as an ideal by $I\cap C_0(G^{(0)})$.
	Since $C_0(G^{(0)})$ has the ideal property, $I\cap C_0(G^{(0)})$ is generated as an ideal of $C_0(G^{(0)})$ by its projections. Hence $I$ is contained in the ideal generated by the projections of $I\cap C_0(G^{(0)})$.
\end{proof}
We will now use our results to describe the space of prime ideals $Prime(C_r^*(G))$. Consider the quasi-orbit space $\mathcal{Q}(G)$ of $G$ which is defined to be the quotient of $G^{(0)}$ by the equivalence relation 
$$u\sim v\ \Leftrightarrow\ \overline{Gu}=\overline{Gv},$$
where $Gu=\lbrace r(g)\mid g\in G_u\rbrace$.
\begin{rem}
Let $G$ be an étale groupoid and $\rho:G^{(0)}\rightarrow \mathcal{Q}(G)$ be the quotient map. Then $\rho$ is a continuous open surjection. It follows that $\mathcal{Q}(G)$ is $T_0$ and totally Baire (i.e. each closed subset is a Baire space).
\end{rem}
We also need the following easy Lemma:
\begin{lem}\label{Lemma:QO-space is a subspace of open-inv. subsets}
	Let $G$ be an étale groupoid. Then the assignment $\overline{Gu}\mapsto G^{(0)}\setminus\overline{Gu}$ defines a continuous injective map $i:\mathcal{Q}(G)\rightarrow \mathcal{O}(G^{(0)})$, which is a homeomorphism onto its image. Hence we can identify $\mathcal{Q}(G)$ with the quotient topology with a subspace of $\mathcal{O}(G^{(0)})$.
\end{lem}
\begin{proof}
	It is obvious that $i$ is injective, hence we only show that it is a homeomorphism onto its image. Recall that $\rho:G^{(0)}\rightarrow \mathcal{Q}(G)$ is the quotient map and a subbase for the topology of $\mathcal{O}(G^{(0)})$ is given by the sets $\mathcal{O}_W=\lbrace V\in \mathcal{O}(G^{(0)})\mid W\not\subseteq V\rbrace$, where $W$ is an open subset of $G^{(0)}$.
	
	We start by proving continuity of $i$: It is clearly enough to show that $i^{-1}(\mathcal{O}_W)=\lbrace \overline{Gu}\mid W\not\subseteq G^{(0)}\setminus\overline{Gu}\rbrace$ is open for every open subset $W\subseteq G^{(0)}$. As $\mathcal{Q}(G)$ carries the quotient topology this happens if and only if $\rho^{-1}(i^{-1}(\mathcal{O}_W))=\lbrace u\in G^{(0)}\mid W\not\subseteq G^{(0)}\setminus\overline{Gu}\rbrace$ is open, if and only if $\lbrace u\in G^{(0)}\mid W\subseteq G^{(0)}\setminus \overline{Gu}\rbrace$ is closed.
	Let $(u_\lambda)_\lambda$ be a net with $W\subseteq G^{(0)}\setminus \overline{Gu_\lambda}$ for all $\lambda$ and $u_\lambda\rightarrow u$ for some $u\in G^{(0)}$. We have to show that $W\subseteq G^{(0)}\setminus \overline{Gu}$ or equivalently $\overline{Gu}\subseteq G^{(0)}\setminus W$. To this end let $v\in Gu$. Then there exists a $g\in G_u$ such that $r(g)=v$. Since $d$ is an open surjection we can apply \cite[Proposition 1.15]{MR2288954} to (after passing to a subnet and relabeling if necessary) find $g_\lambda\in G_{u_\lambda}$ with $g_\lambda\rightarrow g$. Thus $v=r(g)=\lim_{\lambda}r(g_\lambda)\in G^{(0)}\setminus W$.
	As $v\in Gu$ was arbitrary it follows that $Gu\subseteq G^{(0)}\setminus W$ and hence also $\overline{Gu}\subseteq G^{(0)}\setminus W$ as desired.\\
	To see that $i$ is a homeomorphism onto its image let $U\subseteq \mathcal{Q}(G)$ be an open subset. Then $V:=\rho^{-1}(U)$ is open and invariant in $G^{(0)}$. We will show that $i(U)=\mathcal{O}_V \cap i(\mathcal{Q}(G))$, from which the result follows. Let $\overline{Gu}\in U$ be given. Then $V\not\subseteq G^{(0)}\setminus\overline{Gu}$ (if it was, then in particular $u\not\in V$, i.e. $\overline{Gu}=\rho(u)\not\in U$, which contradicts our assumption) and hence $i(\overline{Gu})\in \mathcal{O}_V$.\\
	If conversely $G^{(0)}\setminus \overline{Gu}\in \mathcal{O}_V$ then $\overline{Gu}\in U$ (if it was not, then $u\not\in V$ and hence $\overline{Gu}\subseteq G^{(0)}\setminus V$, which contradicts our assumption).
\end{proof}

\begin{thm}\label{Theorem:Prime Ideals}
Let $G$ be an étale groupoid. If $G$ is inner exact and has the residual intersection property, then the space of prime ideals $Prime(C_r^*(G))$ is homeomorphic to the quasi-orbit space $\mathcal{Q}(G)$.
\end{thm}
\begin{proof}
Consider the map $\Xi:\mathcal{O}(G^{(0)})\rightarrow \mathcal{I}(C_r^*(G))$ defined in the discussion prior to Lemma \ref{Lem:Surjectivity}. In view of Lemma \ref{Lemma:QO-space is a subspace of open-inv. subsets} it restricts to a continuous map $\Psi:\mathcal{Q}(G)\rightarrow \mathcal{I}(C_r^*(G))$.
Let us show that the range of $\Psi$ is contained in $Prime(C_r^*(G))$.
For $u\in G^{(0)}$ write $I_u:=\Psi(\overline{G_u})$ for brevity.
We have to check that $I_u$ is a prime ideal. Let $I,J\in\mathcal{I}(C_r^*(G))$ such that $I\cap J\subseteq I_u$. If $U_I$ and $U_J$ denote the corresponding open $G$-invariant subsets of $G^{(0)}$ this implies $C_0(U_I)\cap C_0(U_J)\subseteq C_0(G^{(0)}\setminus \overline{Gu})$ and hence $U_I\cap U_J\subseteq G^{(0)}\setminus \overline{Gu}$. Thus $u\in G^{(0)}\setminus U_I\cup G^{(0)}\setminus U_J$. Say $u\in G^{(0)}\setminus U_I$. Then $\overline{Gu}\subseteq G^{(0)}\setminus U_I$ as $G^{(0)}\setminus U_I$ is closed and invariant. Hence $U_I\subseteq G^{(0)}\setminus \overline{Gu}$ which implies $I=Ideal_{C_r^*(G)}[C_0(U_I)]\subseteq I_u$.
Hence we can restrict the range and view $\Psi$ as a continuous map
$$\Psi:\mathcal{Q}(G)\rightarrow Prime(C_r^*(G)).$$
We claim that $\Psi$ is the desired homeomorphism.
Injectivity follows easily from Lemma \ref{Lem:Surjectivity}: If $I_u=I_v$ then $G^{(0)}\setminus \overline{Gu}=\Theta(I_u)=\Theta(I_v)=G^{(0)}\setminus \overline{Gv}$ and hence $\overline{Gu}=\overline{Gv}$ as claimed.\\
It remains to show surjectivity. Let $I$ be a prime ideal in $C_r^*(G)$ and $U_I$ be the corresponding invariant open subset of $G^{(0)}$. Let $\rho:G^{(0)}\rightarrow \mathcal{Q}(G)$ be the quotient map and let $F:=\rho(G^{(0)}\setminus U_I)$. We will show that $F$ is the closure of a single point. Since $\mathcal{Q}(G)$ is $T_0$ it will follow that $F$ is a single point and hence that $G^{(0)}\setminus U_I=\overline{Gu}$ for some $u\in G^{(0)}$.
Following Green's argument in \cite[Lemma, p.222]{G78} $F$ is the closure of a single point if and only if $F$ is not a union of two proper closed subsets. Suppose $C_1$ and $C_2$ are proper closed subsets of $F$ such that $F=C_1\cup C_2$. Since $G^{(0)}\setminus U_I$ is saturated, $A_i:=\rho^{-1}(C_i)$, $i=1,2$ is a closed $G$-invariant subset of $G^{(0)}\setminus U_I$ such that $A_1\cup A_2=G^{(0)}\setminus U_I$. If $I_i:=Ideal_{C_r^*(G)}[C_0(G^{(0)}\setminus A_i)]$, then this translates to $I_1\cap I_2=I$.
Now we use that $I$ is prime to obtain $I=I_1$ or $I=I_2$. This contradicts the assumption that $C_1,C_2$ are both proper subsets of $F$.\\
It is clear from what we have shown so far, that the inverse of $\Psi$ is just the restriction of the map $\Theta:\mathcal{I}(C_r^*(G))\rightarrow \mathcal{O}(G^{(0)})$ , hence it is continuous as well, completing the proof.
\end{proof}

If $G$ is second countable, then $\mathcal{Q}(G)$ is second countable since the canonical quotient map is open. Moreover, $C_r^*(G)$ is separable if $G$ is second countable. It is well known that for separable $C^*$-algebras an ideal is prime if and only if it is primitive. Thus the following corollary is an immediate consequence of Theorem \ref{Theorem:Prime Ideals}:
\begin{cor}
	Let $G$ be a second countable étale groupoid. If $G$ is inner exact and has the residual intersection property, then the space of primitive ideals $Prim(C_r^*(G))$ is homeomorphic to the quasi-orbit space $\mathcal{Q}(G)$.
\end{cor}

We end this section with some observations on the notions of inner exactness and the residual intersection property by example. The notion of inner exactness is still very much mysterious, since a characterization of it at the level of the groupoid is still lacking. However we can obtain a quite satisfying answer in the case of so called \textit{coarse groupoids}.

\subsection*{Groupoids arising from coarse geometry}
Let $(X,d)$ be a discrete metric space with bounded geometry, i.e. for every $r>0$, there exists an $n\in\mathbb{N}$ such that every set of diameter $\leq r$ contains at most $n$ elements. For every $R>0$, we consider the $R$-neighbourhood of the diagonal in $X\times X$: $\Delta_R:=\{(x,y)\in X\times X: d(x,y)\leq R\}$. Define
\begin{align*}
G(X):=\bigcup_{R>0}\overline{\Delta}_R\subseteq \beta(X\times X).
\end{align*}
It turns out that the domain, range, inversion and multiplication maps on the pair groupoid $X\times X$ have unique continuous extensions to $G(X)$. With respect to these extensions, $G(X)$ becomes a principal, étale, locally compact $\sigma$-compact Hausdorff topological groupoid with the unit space $\beta X$ (see \cite[Proposition~3.2]{STY02} or \cite[Theorem~10.20]{Roe03}). Since $G(X)^{(0)}=\beta X$ is totally disconnected, $G(X)$ is also ample. Moreover, the uniform Roe algebra $C_u^*(X)$ of $X$ is naturally isomorphic to the reduced groupoid $C^*$-algebra of $G(X)$ (see \cite[Proposition~10.29]{Roe03} for a proof).

In \cite{CW04} the authors establish a one-to-one correspondence between ideals in the uniform Roe algebra $C_u^*(X)$ in which controlled propagation operators are dense and the invariant open subsets of $G(X)^{(0)}=\beta X$. As a consequence, if the metric space $(X,d)$ has Yu's property $(A)$, then all the ideals in $C_u^*(X)$ are of this form and hence there is an isomorphism of lattices between the ideals in $C_u^*(X)$ and the open invariant subsets of $\beta X$. In particular, Yu's property $(A)$ implies that  $\ell^\infty(X)$ separates the ideals of $C_u^*(X)$ by Theorem~\ref{Thm:IdealStructure}. Using our characterization of ideals in reduced groupoid $C^*$-algebras we obtain that the following properties are equivalent: 
\begin{thm}\label{inner-exact property A}
Let $(X,d)$ be a discrete metric space with bounded geometry. Then the following are equivalent:
	\begin{enumerate}
		\item $G(X)$ is inner exact.
		\item $\ell^\infty(X)$ separates the ideals of $C_r^*(G(X))$.
		\item $X$ has property $(A)$.
		\item $G(X)$ is amenable.
	\end{enumerate}
\end{thm}
\begin{proof}
$(1)\Rightarrow (2)$: Since $G(X)$ is a principal groupoid, it is automatically essentially principal. Hence, $(1)\Rightarrow (2)$ follows from Corollary~\ref{essprin+inner exact}.\\
$(2)\Rightarrow (3)$: It follows from Theorem~1.3 in \cite{MR3146831} that we only have to show that all ghost operators are compact. Let $I$ be the ideal consisting of all ghost operators in $C_u^*(X)\cong C_r^*(G(X))$, then we have $I=Ideal_{C_r^*(G(X))}[I\cap \ell^\infty(X)]$ by \cite[Theorem~6.4 and Remark~6.5]{CW04} or Lemma~\ref{Lemma:CharacterizationsOfIdealSeparation}. Since $I\cap \ell^\infty(X)\subseteq K(\ell^2(X))$, we conclude that all ghost operators are compact.\\
$(3)\Leftrightarrow (4)$: It follows from \cite[Theorem~5.3]{STY02}.\\
$(4)\Rightarrow (1)$: It follows easily from \cite[Proposition~5.1.1, Proposition~6.1.8 and Lemma~6.3.2]{MR1799683}.
\end{proof}
The above theorem provides a negative answer to the following conjecture by Sierakowski.
\begin{conjecture}{\cite[Conjecture~0.1]{Sierakowski10}}
Let $(A,\Gamma)$ be a $C^*$-dynamical system with $\Gamma$ discrete. If the action of $\Gamma$ on $\hat{A}$ is essentially free, then $A$ separates the ideals of $A\rtimes_r \Gamma$.
\end{conjecture}
It is well-known that for a discrete countable group $\Gamma$, the uniform Roe algebra $C^*_u(\Gamma)$ is $*$-isomorphic to the reduced crossed product $C^*$-algebra $\ell^\infty(\Gamma)\rtimes_r \Gamma$, where $\Gamma$ acts on $\ell^\infty(\Gamma)$ by right-translations (see e.g. \cite[Proposition 5.1.3]{BO08}). Moreover, the action of $\Gamma$ on $\beta \Gamma$ is always (essentially) free and $\ell^\infty(\Gamma)$ separates the ideals of $\ell^\infty(\Gamma)\rtimes_r \Gamma$ if and only if $\Gamma$ is an exact group by Theorem~\ref{inner-exact property A} (see also \cite{MR1739727}). On the other hand, there exist finitely generated
non-exact groups \cite{MR1978492, Osa}.\\
\\
A very different class of examples arises from higher rank graphs. In \cite{MR1745529} Kumjian and Pask introduce the notion of a $k$-graph $\Lambda$ and construct a $C^*$-algebra $C^*(\Lambda)$. They also associate an ample groupoid $G_\Lambda$ to the $k$-graph such that $C^*(G_\Lambda)\cong C^*(\Lambda)$. The groupoid $G_\Lambda$ is always amenable, hence inner exactness is automatic in that case. The groupoid $G_\Lambda$ has the residual intersection property if and only if $G_\Lambda$ is essentially principal by Proposition \ref{Prop:TopPrincipalAndNontrivialIntersections}. It is well-known that $G_\Lambda$ is essentially principal if and only if $\Lambda$ is strongly aperiodic (see \cite{MR3189779} for a definition of strong aperiodicity and \cite[Proposition~4.5]{MR1745529} for a proof). Hence, our Theorem \ref{Thm:IdealStructure} recovers \cite[Corollary~3.9]{PSS16}.

\section{Pure infiniteness and paradoxicality}
Let us recall the basic definitions of infiniteness of positive elements. For positive elements $a,b$ in a $C^*$-algebra $A$ we say that $a$ is \textit{Cuntz-below} $b$ and write $a\precsim b$ if there exists a sequence $(r_n)_n$ in $A$ such that $r_n^*br_n\rightarrow a$. We extend this to matrices over $A$ as follows: For $a\in M_n(A)^+$ and $b\in M_m(A)^+$ write $a\precsim b$ if there exists a sequence $(r_n)_n$ in $M_{m,n}(A)$ with $r_n^*br_n\rightarrow a$. For $a\in M_n(A)$ and $b\in M_m(A)$ we write $a\oplus b$ for the element $diag(a,b)\in M_{n+m}(A)$.

A positive element $a\in A$ is called \textit{infinite} if there exists a positive element $0\neq b\in A$ such that $a\oplus b\precsim a$. It is called \textit{properly infinite} if it is non-zero and $a\oplus a\precsim a$. This extends the usual concepts of infinite and properly infinite projections (see Section 3 in \cite{KR00}). A $C^*$-algebra $A$ is \textit{purely infinite} if every non-zero positive element of $A$ is properly infinite (see \cite[Theorem~4.16]{KR00}).

Using our results on the ideal structure of $C_r^*(G)$ we can generalize many results from \cite{BCS15} to the non-minimal setting. In particular, the following proposition extends \cite[Theorem~3.3]{BCS15} to the non-minimal setting and the idea of the proof comes from \cite[Proposition~2.1]{RS12}.

\begin{prop}
Let $G$ be an étale groupoid and $E:C_r^*(G)\rightarrow C_0(G^{(0)})$ be the canonical faithful conditional expectation. Suppose that $C_0(G^{(0)})$ separates the ideals of $C_r^*(G)$. Then $C_r^*(G)$ is purely infinite if and only if all non-zero positive elements in $C_0(G^{(0)})$ are properly infinite in $C_r^*(G)$ and $E(a)\precsim a$ for all positive elements $a$ in $C_r^*(G)$.
\end{prop}
\begin{proof}
If $C_r^*(G)$ is purely infinite, then every non-zero positive element in $C_r^*(G)$ is properly infinite by \cite[Theorem~4.16]{KR00}. Let $a$ be any non-zero positive element in $C^*_r(G)$. It follows from Theorem~\ref{Thm:IdealStructure}, Proposition~\ref{Prop:IdealInclusionAndResidualIntersection} and Lemma~\ref{inner exact and intersection} that $E(a)$ belongs to the ideal in $C_r^*(G)$ generated by the properly infinite positive element $a$. Hence, we have $E(a)\precsim a$ by \cite[Proposition~3.5 (ii)]{KR00}.

Conversely, it suffices to show that every non-zero positive element $a$ in $C_r^*(G)$ is properly infinite by \cite[Theorem~4.16]{KR00}. From Theorem~\ref{Thm:IdealStructure} and Lemma~\ref{Lemma:IdealInclusionAndInnerExactness} we have that $a$ belongs to the ideal in $C_r^*(G)$ generated by $E(a)$. Since $E(a)$ is a non-zero positive element in $C_0(G^{(0)})$, it is properly infinite in $C_r^*(G)$ by the assumption. Moreover, \cite[Proposition~3.5 (ii)]{KR00} implies that $a\precsim E(a)$. We conclude that $a$ is properly infinite by the following easy computation
$$
a\oplus a \precsim E(a) \oplus E(a)  \precsim E(a)  \precsim a.
$$
Thus, $C_r^*(G)$ is purely infinite.
\end{proof}

 Note that thanks to Corollary \ref{Cor:IdealSeparationImplies(IP)}, the assumptions of the following theorem imply that $C_r^*(G)$ has the ideal property (IP). Consequently, $C_r^*(G)$ is purely infinite if and only if it is strongly purely infinite if and only if it is weakly purely infinite by \cite[Proposition 2.14]{PR07}. 
\begin{thm}\label{Thm:CharacterizationPurelyInfinite}
	Suppose $G$ is an ample groupoid, such that $G$ is inner exact and essentially principal. Then the following are equivalent:
	\begin{enumerate}
		\item $C_r^*(G)$ is (strongly) purely infinite.
		\item Every non-zero projection $p\in C_0(G^{(0)})$ is properly infinite in $C_r^*(G)$.
		\item Every non-zero projection in $C_0(D)$ is infinite in $C_r^*(G_D)$ for every closed $G$-invariant subset $D\subseteq G^{(0)}$.
	\end{enumerate}
\end{thm}
\begin{proof}
	$(1)\Rightarrow (2)$: If $C_r^*(G)$ is purely infinite, then every non-zero projection in $C_r^*(G)$ is properly infinite.\\
	
	$(2)\Rightarrow (3)$: Let $D\subseteq G^{(0)}$ be a closed $G$-invariant subset and $0\neq p\in C_0(D)$ be a projection. Then there is a compact open set $A\subseteq D$ such that $p=1_A$. Consequently $A=D\cap B$ for some open set $B$ in $G^{(0)}$. Since $G^{(0)}$ is totally disconnected, we can assume that $B$ is also compact (if it is not, write $B=\bigcup_i B_i$ with $B_i$ compact open. As $A$ is compact, we can reindex to write $A\subseteq \bigcup_{i=1}^n B_i$. Then $A=D\cap \bigcup_{i=1}^n B_i$, where $\bigcup_{i=1}^n B_i$ is compact open). Consequently $1_B$ is a non-zero projection in $C_0(G^{(0)})$ and hence properly infinite in $C_r^*(G)$.
	Let $\pi:C_r^*(G)\rightarrow C_r^*(G_D)$ be the quotient map. Then $\pi(1_B)=1_A$ and hence $1_A$ is properly infinite.\\
	
	$(3)\Rightarrow (1)$: By \cite[Proposition~4.7]{KR00} it is enough to show that every non-zero hereditary sub-$C^*$-algebra in every quotient of $C_r^*(G)$ contains an infinite projection. So let $I\subseteq C_r^*(G)$ be an ideal. Let $U\subseteq G^{(0)}$ be the open $G$-invariant set corresponding to the ideal $I\cap C_0(G^{(0)})$ and $D=G^{(0)}\setminus U$.
	By Corollary~\ref{essprin+inner exact} $C_0(G^{(0)})$ separates the ideals of $C_r^*(G)$ and hence we have $C_r^*(G)/I=C_r^*(G_D)$. Now let $B\subseteq C_r^*(G_D)$ be a non-zero hereditary sub-$C^*$-algebra and $0\neq b\in B$ a positive element. By \cite[Lemma 3.2]{BCS15} there exists a positive element $0\neq h\in C_0(D)$ such that $h\precsim b$ in $C_r^*(G_D)$. The hereditary sub-$C^*$-algebra $\overline{hC_0(D)h}$ contains a projection $p$ since $G^{(0)}$ is totally disconnected. Applying $(3)$ now, we see that $p$ is infinite in $C_r^*(G_D)$. By \cite[Proposition~2.7 (i)]{KR00} we get $p\precsim h\precsim b$. Since $p$ is a projection we can find $x\in C_r^*(G_D)$ such that $p=x^*bx$ (cf. \cite[Proposition~2.6]{KR00}). If we set $z:=b^{\frac{1}{2}}x$, then $z^*z=p$ and $q:=zz^*=b^{\frac{1}{2}}xx^*b^{\frac{1}{2}}$ is a projection in $B$, which is equivalent to $p$, hence also infinite.
\end{proof}

We can also use our findings to slightly improve the results in \cite{BCS15} concerning strong pure infinitness to étale groupoids without the weak containment property. For a definition of a filling family and the matrix diagonalization property we refer the reader to \cite{KS15}.
\begin{prop}\label{Proposition:Filling family}
	Let $G$ be an étale groupoid. If $G$ is inner exact and essentially principal, then $C_0(G^{(0)})^+$ is a filling family for $C_r^*(G)$ in the sense of \cite{KS15}.
\end{prop}
\begin{proof}
	It follows from the assumptions and Corollary \ref{essprin+inner exact} that $C_0(G^{(0)})$ separates the ideals in $C_r^*(G)$ and hence the same proof as in \cite[Proposition 6.3]{BCS15} works.
\end{proof}

Using this, we obtain the following:
\begin{cor}
	Let $G$ be an étale, inner exact and essentially principal groupoid. Then the following are equivalent:
	\begin{enumerate}
		\item $C_r^*(G)$ is strongly purely infinite.
		\item Every pair $(f,g)\in C_0(G^{(0)})^+\times C_0(G^{(0)})^+$ has the matrix diagonalization property.
	\end{enumerate}
\end{cor}
\begin{proof}
	$(1)\Rightarrow (2)$: This is a direct consequence of \cite[Lemma~5.8]{KR02}.
	
	$(2)\Rightarrow (1)$: This follows from combining Proposition \ref{Proposition:Filling family} with \cite[Theorem 1.1]{KS15}.
\end{proof}

Let us return to the totally disconnected setting.
In view of Theorem \ref{Thm:CharacterizationPurelyInfinite} it would be desirable to have a condition at the level of the (ample) groupoid $G$ ensuring that a given non-zero projection $p\in C_0(G^{(0)})$ is properly infinite in $C_r^*(G)$. As $G^{(0)}$ is totally disconnected, every such projection is just the characteristic function of a compact open subset of $G^{(0)}$.
In the classical setting of (partial) actions of discrete groups on totally disconnected spaces such a condition is given by the notion of paradoxical decompositions (see \cite{MR3144248,MR2974211,RS12}).
In the following we generalize this notion to the setting of étale groupoids.

\begin{defi}\label{DefParadoxicalDecomposition}
	Let $G$ be an étale groupoid and $\mathbb{E}$ a family of bisections of $G$. For $k>l>0$ a non-empty subset $A\subseteq G^{(0)}$ is called $(\mathbb{E},k,l)$-paradoxical, if for each $1\leq i\leq k$ there exists a number $n_i\in\N$ and bisections $V_{i,1},\ldots,V_{i,n_i}$ in $\mathbb{E}$ and $m_{i,1},\ldots ,m_{i,n_i}\in\lbrace1,\ldots l\rbrace$, such that:
	\begin{enumerate}
		\item $\bigcup\limits_{j=1}^{n_i} d(V_{i,j})=A$ for all $1\leq i\leq k$, and
		\item the sets $r(V_{i,j})\times \lbrace m_{i,j}\rbrace\subseteq A\times \lbrace 1,\ldots, l\rbrace$ are pairwise disjoint.
	\end{enumerate}
	We call $A$ \textit{completely $\mathbb{E}$-non-paradoxical} if $A$ fails to be $(\mathbb{E},k,l)$-paradoxical for all integers $k>l>0$.
\end{defi}
Some canonical choices for $\mathbb{E}$ could be the set $G^{op}$ of all open bisections or the set of all Borel bisections. The best results however can be achieved in a totally disconnected setting, i.e. if $G$ is ample. In that case we will exclusively use the set $G^a$ of all compact open bisections.

\begin{rem}\label{Remark:paradoxical} Let us make a few easy observations:
\begin{enumerate}
	\item It is not hard to see that if $A\subseteq G^{(0)}$ is $(\mathbb{E},k,l)$-paradoxical for $k>l>0$, then $A$ is $(\mathbb{E},k',l')$-paradoxical for all $k\geq k'>l'\geq l$.
	\item Let $G$ be an ample groupoid. Given a $(G^a,k,l)$-paradoxical decomposition as in the definition, we can always make sure that for each fixed $1\leq i\leq k$ the union $\bigcup_{j=1}^{n_i} d(V_{i,j})$ is disjoint, since set differences of compact open sets are compact open.
	\item It will follow from Remark \ref{Remark:Comparison with old definition of type sgrp} in conjunction with Lemma \ref{Lem:ParadoxicalityInTermsOfTypeSemigrp} that our definition of paradoxical decomposition coincides with the definitions given in \cite{MR3144248,MR2974211,RS12} for the special case of (partial) group actions of discrete groups on totally disconnected spaces.
\end{enumerate}
\end{rem}

\begin{exam}[The Cuntz Groupoid $G_n$] Let us look at a well-known and simple example to illustrate paradoxicality. The definition of the Cuntz groupoid is due to Renault (see \cite{Renault80}), but we will follow the exposition given in \cite{Paterson99}. Fix $n\geq 2$ and consider the set $X=\lbrace 1,\ldots n\rbrace^\N$  of sequences with integer values between $1$ and $n$ equipped with the product topology. This is a totally disconnected compact Hausdorff space. The Cuntz groupoid $G_n$ consists of triples $$(\alpha\gamma,\ell(\alpha)-\ell(\beta),\beta\gamma)\in X\times \Z\times X,$$
where
\begin{enumerate}
	\item $\gamma\in X$,
	\item $\alpha$ and $\beta$ are finite strings with values in $\lbrace 1,\ldots n\rbrace$, and
	\item $\ell(\alpha)$ and $\ell(\beta)$ denote the length of $\alpha$ and $\beta$ respectively.
\end{enumerate}
Define multiplication and inversion on $G_n$ by
$$(\gamma,k,\gamma')(\gamma',k',\gamma'')=(\gamma,k+k',\gamma''),\text{ and } (\gamma,k,\gamma')^{-1}=(\gamma',-k,\gamma).$$
One readily checks that this turns $G_n$ into a groupoid. We can identify $X$ with $(G_n)^{(0)}$ via $\gamma\mapsto (\gamma,0,\gamma)$ and under this identification the domain and range maps are given by $d(\gamma,k,\gamma')=\gamma'$ and $r(\gamma,k,\gamma')=\gamma$. We topologize $G_n$ as follows: For $V\subseteq X$ open and $\alpha,\beta$ finite strings let $$U_{\alpha,\beta,V}=\lbrace (\alpha\gamma,\ell(\alpha)-\ell(\beta),\beta\gamma)\mid \gamma\in V\rbrace.$$
We define a topology on $G_n$ by declaring the sets $U_{\alpha,\beta,V}$ to be a basis for the topology. It is then straightforward to check that $G_n$ is an ample groupoid and $U_{\alpha,\beta,V}\in G_n^a$. Note that this topology restricts to the product topology on $X=G_n^{(0)}$ but is different from the topology inherited from the product topology on $X\times\Z\times X$. If we set $U_{\alpha,\beta}=U_{\alpha,\beta,X}$ then actually the $U_{\alpha,\beta}$ are enough to form a basis for the topology of $G_n$ since $U_{\alpha,\beta,V}=(\alpha V)U_{\alpha,\beta}$ and the open sets $\alpha X$ form a basis for the topology of $X$.\\
We claim that every compact open subset of the form $\alpha X=U_{\alpha,\alpha}$ is $(G^a,2,1)$-paradoxical.
For $1\leq i\leq n$ let $\beta_i$ be the finite string $\alpha i$. Then $d(U_{\beta_i,\alpha})=\alpha X$ and $r(U_{\beta_i,\alpha})\subseteq \alpha X$ for all $1\leq i\leq n$. Moreover the sets $r(U_{\beta_i,\alpha})$ are pairwise disjoint whence the claim follows. Note that in this situation we observe an even stronger form of paradoxicality since we actually have $\bigsqcup_{i=1}^n r(U_{\beta_i,\alpha})=\alpha X$.
\end{exam}

\subsection*{Pseudogroups and metric spaces}	
Let $X$ be a totally disconnected compact Hausdorff space and $\mathcal{G}$ be a pseudogroup on $X$ in the sense of \cite[Definition 1.]{MR1721355} consisting of partial homeomorphisms $\varphi:dom(\varphi)\rightarrow ran(\varphi)$, where $dom(\varphi)$ and $ran(\varphi)$ are clopen subsets of $X$. We will assume that $X=\bigcup_{\varphi\in\mathcal{G}}dom(\varphi)$. Given such a pseudogroup, Skandalis, Tu and Yu in \cite{STY02} construct an ample groupoid $G(\mathcal{G})$ with unit space $X$.
A basis for the topology is given by a family $(U_\varphi)_{\varphi\in\mathcal{G}}$ of compact open bisections of $G(\mathcal{G})$ such that $d(U_\varphi)=dom(\varphi)$, $r(U_\varphi)=ran(\varphi)$ and the canonical map $\alpha_{U_\varphi}:d(U_\varphi)\rightarrow r(U_\varphi)$ coincides with $\varphi$.
We say that $X$ is \textit{weakly $\mathcal{G}$-paradoxical} if there exist $\varphi_1,\varphi_2\in\mathcal{G}$ with $dom(\varphi_1)=X=dom(\varphi_2)$ and $ran(\varphi_1)\cap ran(\varphi_2)=\emptyset$.

\begin{prop}\label{Proposition: ParadoxicalPseudogroups}
	Let $X$ be a totally disconnected compact Hausdorff space and $\mathcal{G}$ be a pseudogroup on $X$ such that $X=\bigcup_{\varphi\in\mathcal{G}}dom(\varphi)$. Then $X$ is weakly $\mathcal{G}$-paradoxical if and only if $X$ is $(G(\mathcal{G})^a,2,1)$-paradoxical.
	
	Moreover, if $X$ is hyperstonian then $X$ admits a $\mathcal{G}$-paradoxical decomposition in the sense of \cite[Definition~5]{MR1721355} if and only if $X$ is weakly $\mathcal{G}$-paradoxical.
\end{prop}
\begin{proof}
	If $X$ admits a $\mathcal{G}$-paradoxical decomposition, there exist $\varphi_1,\varphi_2\in\mathcal{G}$ such that $dom(\varphi_i)=X$ and $ran(\varphi_1)\cap ran(\varphi_2)=\emptyset$. Then the compact open bisections $U_{\varphi_i}$, $i=1,2$ obviously implement a $(G(\mathcal{G})^a,2,1)$-paradoxical decomposition.\\
	Conversely, assume that $X$ is $(G(\mathcal{G})^a,2,1)$-paradoxical. Then we find $n,m\in\N$ and compact open bisections $V_1,\ldots,V_{n+m}$ such that
	$$X=\bigsqcup\limits_{i=1}^n d(V_i)=\bigsqcup\limits_{i=n+1}^{n+m} d(V_i)$$
	and the $r(V_i)$ are pairwise disjoint. Since the domains and ranges of the bisections within each decomposition are disjoint, the unions $S_1:=\bigsqcup_{i=1}^n V_i$ and $S_2:=\bigsqcup_{i=n+1}^{n+m} V_i$ are compact open bisections in their own right with the property that
	$X=d(S_1)=d(S_2)$ and $r(S_1)\cap r(S_2)=\emptyset$.
	We would like to show that the partial homeomorphisms $\alpha_{S_i}:X=d(S_i)\rightarrow r(S_i)$ induced by $S_i$ are contained in $\mathcal{G}$.
	To this end use compactness of $S_i$ to find a partition $S_i=\bigsqcup_{j=1}^{n_i}U_{\varphi_{i,j}}$. Then $(\alpha_{S_i})_{\mid dom(\varphi_{i,j})}=\alpha_{U_{\varphi_{i,j}}}=\varphi_{i,j}\in \mathcal{G}$ for each $1\leq j\leq n_i$ and $i=1,2$. Consequently, we have $\alpha_{S_i}\in\mathcal{G}$ by \cite[Definition~1. (v)]{MR1721355}. 
	
	The last statement follows from a classical Bernstein-Schr\"oder-type argument (see e.g. \cite[Remark~2.10]{ALLW16}), where we use the fact that countably infinite disjoint unions of compact open sets are compact open in a hyperstonian space.
\end{proof}

An interesting class of examples arises from coarse geometry. The following Corollary provides a groupoid picture of the main theorem of \cite{ALLW17}.

\begin{cor} Let $X$ be a discrete metric space with bounded geometry. Then the following are equivalent:
	\begin{enumerate}
		\item The metric space $X$ is paradoxical in the sense of \cite[Definition~2.9]{ALLW16}.
		\item The space $\beta X$ is $\mathcal{G}(X)$-paradoxical, where $\mathcal{G}(X)$ is the pseudogroup of partial transformations on $\beta X$ in the sense of \cite[Definition~3.1]{STY02}.
		\item $\beta X=G(X)^{(0)}$ is $(G(X)^a,2,1)$-paradoxical.
	\end{enumerate}	
\end{cor}
\begin{proof}
	$(1)\Leftrightarrow (2)$: This follows from the fact that there is a one-to-one correspondence between subsets of $X$ and compact open subsets of $\beta X$ given by $A\mapsto \overline{A}\subseteq \beta X$ for all $A\subseteq X$.
	
	$(2)\Leftrightarrow (3)$: Since $\beta X$ is hyperstonian, this follows from Proposition~\ref{Proposition: ParadoxicalPseudogroups} and the fact that the coarse groupoid $G(X)$ is isomorphic to $G(\mathcal{G}(X))$ (see \cite[Proposition~3.2]{STY02}).	
\end{proof}


Let us now see how we can employ paradoxicality to produce (properly) infinite projections in $C_r^*(G)$:
Recall that a projection $p$ in a $C^*$-algebra $A$ is \textit{properly infinite} if and only if there exist partial isometries $x,y\in A$ such that $x^*x=y^*y=p$ and $xx^*+yy^*\leq p$.

\begin{prop}\label{Prop:ParadSubsetsGivePropInfiniteProjections}
	Let $G$ be an étale groupoid and $A\subseteq G^{(0)}$ a compact open subset. If $A$ is $(G^{op},2,1)$-paradoxical, then there exist elements $f,g\in C_c(G)$ with $E(f),E(g)\geq 0$ such that $f^*\ast f=g^*\ast g=1_A$ and $f\ast f^*+g\ast g^*\leq 1_A$. In this case $1_A$ is a properly infinite projection in $C_r^*(G)$.\\
\end{prop}
\begin{proof}
	Since $A$ is $(G^{op},2,1)$-paradoxical, we can find open bisections $V_1,\ldots ,V_{n+m}\in G^{op}$ such that $\bigcup_{i=1}^n d(V_i)=\bigcup_{i=n+1}^{n+m} d(V_i)=A$, and the $(r(V_i))_i$ are pairwise disjoint subsets of $A$. Then we can find partitions of unity $(\varphi_i)_{i=1}^n$ and $(\varphi_i)_{i=n+1}^{n+m}$ subordinate to the two coverings of $A$ respectively with $supp(\varphi_i)\subseteq d(V_i)$.
	Set $\psi_i=(\varphi^{\frac{1}{2}}_i\circ d_{\mid V_i})\in C_c(V_i)\subseteq C_c(G)$ and define
	\[f=\sum\limits_{i=1}^n \psi_i,\ \ g=\sum\limits_{i=n+1}^{n+m}\psi_i\]
	Note that from the fact, that the $r(V_i)$ are pairwise disjoint we obtain $V_i^{-1}V_j=\emptyset$ and thus $\psi_i^*\ast\psi_j=0$ for all $i\neq j$. For $i=j$ we have $V_i^{-1}V_i=d(V_i)\subseteq G^{(0)}$ and thus $supp(\psi_i^*\ast\psi_i)\subseteq d(V_i)\subseteq G^{(0)}$ and for $u\in G^{(0)}$ we compute

\[	\psi_i^*\ast\psi_i(u)=\sum\limits_{g\in G^u}\psi_i(g^{-1})^2=\sum\limits_{g\in G_u\cap V_i}\psi_i(g)^2 = \varphi_i(u)\]

Using this we can compute for all $g\in G$:
\begin{align*}
f^*\ast f(g)& =\sum\limits_{h\in G^{r(g)}}f(h^{-1})f(h^{-1}g)\\
& =\sum\limits_{h\in G^{r(g)}}\sum\limits_{i,j=1}^n\psi_i(h^{-1})\psi_j(h^{-1}g)\\
&=\sum\limits_{i,j=1}^n\psi_i^*\ast\psi_j(g)\\
& =\sum\limits_{i=1}^n \psi_i^*\ast\psi_i(g)
\end{align*}
The last term is non-zero if and only if $g=u\in A$ and then by the above computation $f^*\ast f(u)=\sum_{i=1}^n \varphi_i(u)=1$. Thus we have shown $f^*\ast f=1_A$. Similarly, we can show $g^*\ast g=1_A$ and $g^*\ast f=0$, and hence $f\ast f^*+g\ast g^*\leq 1_A$.
\end{proof}

The ideas in the proof of this proposition can also be used in the case of a $(G^a,k,l)$ paradoxical decomposition for $k>l>1$ to obtain infinite projections in a matrix algebra over $C_r^*(G)$. Recall that a projection $p\in A$ is \textit{infinite} if there exists a proper subprojection $q<p$ such that $p\sim q$.
\begin{prop}\label{Prop:Stably finite implies non-paradoxical}
	Let $G$ be an ample groupoid and $A\subseteq G^{(0)}$ a compact open subset. Suppose $A$ is $(G^a,k,l)$-paradoxical. Then $1_k\otimes 1_A$ is an infinite projection in $M_k(C_r^*(G))$.
\end{prop}
\begin{proof}
	Suppose there exists a compact open subset $A\subseteq G^{(0)}$ such that $A$ is $(G^a,k,l)$-paradoxical for some $k>l>0$. Then for each $1\leq i\leq k$ there exists a number $n_i\in\N$ and bisections $V_{i,1},\ldots,V_{i,n_i}$ in $G^a$ and $m_{i,1},\ldots ,m_{i,n_i}\in\lbrace1,\ldots l\rbrace$, such that $\bigsqcup_{j=1}^{n_i} d(V_{i,j})=A$ for all $1\leq i\leq k$, and
	the sets $r(V_{i,j})\times \lbrace m_{i,j}\rbrace\subseteq A\times \lbrace 1,\ldots, l\rbrace$ are pairwise disjoint.
	Put $$a_{i,j}=e_{m_{i,j},i}\otimes 1_{V_{i,j}}\in M_k(C_r^*(G)),$$
	Then $a_{i,j}^*a_{i',j'}$ is zero unless $(i,j)=(i',j')$ and in that case we have $a_{i,j}^*a_{i,j}=e_{i,i}\otimes 1_{d(V_{i,j})}$. Similarly
	$a_{i,j}a_{i',j'}^*$ is zero unless $(i,j)=(i',j')$ and then $a_{i,j}a_{i,j}^*=e_{m_{i,j},m_{i,j}}\otimes 1_{r(V_{i,j})}$.
	Consequently the $a_{i,j}$ are partial isometries with pairwise orthogonal domain and range projections.
	We compute
	$$\sum\limits_{i=1}^k\sum\limits_{j=1}^{n_i}a_{i,j}^*a_{i,j}=\sum\limits_{i=1}^k e_{i,i}\otimes \left(\sum\limits_{j=1}^{n_i} 1_{d(V_{i,j})}\right)=\sum\limits_{i=1}^k e_{i,i}\otimes 1_A=1_k\otimes 1_A,$$
	and similarly, using a re-indexation we get
	$$\sum\limits_{i=1}^k\sum\limits_{j=1}^{n_i}a_{i,j}a_{i,j}^*=\sum\limits_{i=1}^k\sum\limits_{j=1}^{n_i} e_{m_{i,j},m_{i,j}}\otimes 1_{r(V_{i,j})}= \sum\limits_{r=1}^l e_{r,r}\otimes (\sum\limits_{\lbrace r\mid m_{i,j}=r\rbrace}1_{r(V_{i,j})})\leq 1_l\otimes 1_A.$$
	If we put $p:=\sum_{r=1}^l e_{r,r}\otimes (\sum_{\lbrace r\mid m_{i,j}=r\rbrace}1_{r(V_{i,j})})$, then these computations imply
	$$p\leq 1_l\otimes 1_A<1_k\otimes 1_A\sim p.$$
	Hence $1_k\otimes 1_A$ is infinite as desired.
\end{proof}


Let us formulate the conclusions we can draw from Theorem \ref{Thm:CharacterizationPurelyInfinite} in combination with Proposition \ref{Prop:ParadSubsetsGivePropInfiniteProjections}. 
\begin{cor}\label{Cor: Every set is paradoxical implies purely infinite}
	Let $G$ be an ample groupoid, which is essentially principal and inner exact. Let $\mathcal{B}$ be a basis for the topology of $G^{(0)}$ consisting of compact open sets. Suppose each $A\in \mathcal{B}$ is $(G^a,2,1)$-paradoxical. Then $C_r^*(G)$ is (strongly) purely infinite.
	
\end{cor}
\begin{proof}
	From Proposition \ref{Prop:ParadSubsetsGivePropInfiniteProjections} we conclude that the characteristic functions $1_A\in C_0(G^{(0)})$ for $A\in\mathcal{B}$ are properly infinite. We use this to check condition $(3)$ of Theorem \ref{Thm:CharacterizationPurelyInfinite}: Let $D\subseteq G^{(0)}$ be a closed invariant subset and $0\neq p\in C_0(D)$ a projection. Hence $p=1_U$ for some compact open subset $U\subseteq D$. By assumption there exists a compact open subset $A\in\mathcal{B}$ such that $A\cap D\subseteq U$ and the projection $1_A$ is properly infinite in $C_r^*(G)$. Hence its image under the quotient map $\pi:C_r^*(G)\rightarrow C_r^*(G_D)$ is infinite in $C_r^*(G_D)$. But $\pi(1_A)=1_{A\cap D}\leq p$ and thus $p$ is infinite.
\end{proof}

We remark, that if in the situation of the previous corollary $G$ is moreover assumed to be amenable, then $dim_{nuc}(C_r^*(G))=1$.
Indeed, note that by Corollary \ref{Cor:IdealSeparationImplies(IP)} $C_r^*(G)$ has the ideal property (IP) and hence $Prim(C_r^*(G))$ has a basis for its topology consisting of compact open sets (confer \cite[Proposition 2.11]{PR07}). Since $G$ is amenable, we also have that $C_r^*(G)$ is nuclear. It follows that we have $\dim_{nuc}(C_r^*(G))=1$ by an earlier announced result of Gabe, which will appear in a forthcoming paper called "Strong pure infinite $C^*$-algebras have nuclear dimension one" by Bosa, Gabe, Sims, and White.

\section{The type semigroup}

It was already noted by Tarski, who studied paradoxical decompositions in a discrete setting, that the notion of paradoxical decomposition can be phrased in algebraic terms, by introducing the so called type semigroup. In this section we will define a version of the type semigroup in the context of ample groupoids and show that a paradoxical decomposition reflects as an inequality (with respect to the algebraic preorder) in this semigroup (see Lemma \ref{Lem:ParadoxicalityInTermsOfTypeSemigrp}). The importance of understanding the type semigroup becomes apparent when combining it with the following theorem of Tarski, a proof of which can be found in \cite[Theorem~9.1]{Wagon93}.
\begin{thm}(Tarski)\label{Thm:ExistenceOfStatesOnTypeSemigroup}
	Let $S$ be an abelian monoid equipped with the algebraic preorder $\leq$ and let $x\in S$. Then the following are equivalent:
	\begin{enumerate}
		\item $(n+1)x\not\leq nx$ for all $n\in\N$.
		\item There exists an additive map $f:S\rightarrow [0,\infty]$ such that $f(x)=1$.
	\end{enumerate}
\end{thm}
We will see how a map as in $(2)$ of the above theorem on the type semigroup can (under suitable hypothesis on $G$) be lifted to a trace on $C_r^*(G)$.

Let $G$ be an ample groupoid. Recall that $G^a$ denotes the set of all compact open bisections of $G$. We define a relation $\sim$ on the set
$$\lbrace \bigcup\limits_{i=1}^n A_i\times\lbrace i\rbrace\mid n\in\N, A_i\in G^a \textit{ and } A_i\subseteq G^{(0)}\rbrace,$$
as follows: For sets $A=\bigcup_{i=1}^n A_i\times\lbrace i\rbrace$ and $B=\bigcup_{j=1}^m B_i\times\lbrace j\rbrace$ we write $A\sim B$, if there exists an $l\in\N$, bisections $W_1,\ldots,W_l\in G^a$ and numbers $n_1,\ldots,n_l,m_1,\ldots,m_l\in\N$ such that
$$A=\bigsqcup\limits_{k=1}^l d(W_k)\times\lbrace n_k\rbrace,\ \ B=\bigsqcup\limits_{k=1}^l r(W_k)\times\lbrace m_k\rbrace.$$
\begin{lem}
	The relation $\sim$ is an equivalence relation.
\end{lem}
\begin{proof}
If $A=\bigcup_{i=1}^n A_i\times\lbrace i\rbrace$, then let $l:=n$ and $W_i:=A_i\in G^a$. Thus we obtain $A\sim A$. To see symmetry suppose $A\sim B$ via the bisections $W_1,\ldots, W_l\in G^a$. Then $B\sim A$ via $W_1^{-1},\ldots, W_l^{-1}\in G^a$. Transitivity follows by taking common refinements.
\end{proof}
\begin{defi}
	Let $G$ be an ample groupoid. We define the type semigroup of $G$ to be the set
	$$S(G,G^a):=\lbrace \bigcup\limits_{i=1}^n A_i\times\lbrace i\rbrace\mid n\in\N, A_i\in G^a \textit{ and } A_i\subseteq G^{(0)}\rbrace/\sim$$
\end{defi}
As usual we denote by $[A]$ the equivalence class of $A$, and define the addition in $S(G,G^a)$ by
$$\left[  \bigcup\limits_{i=1}^n A_i\times\lbrace i\rbrace\right] + \left[ \bigcup\limits_{j=1}^m B_j\times\lbrace j\rbrace\right]=\left[ \bigcup\limits_{i=1}^n A_i\times\lbrace i\rbrace \cup \bigcup\limits_{j=1}^m B_j\times\lbrace n+j\rbrace\right] $$
By slight abuse of notation we will also write $[A]$ for the class of $A\times \lbrace 1\rbrace$ if $A\subseteq G^{(0)}$ and $A\in G^a$.
The semigroup $S(G,G^a)$ has neutral element $0=[\emptyset]$ and we equip it with the algebraic preorder, i.e. $x\leq y$ if $y=x+z$ for some $z\in S(G,G^a)$.

\begin{rem}\label{Remark:Comparison with old definition of type sgrp}
	Our definition of the type semigroup generalizes the earlier definitions for (partial) group actions of discrete groups: If $((X_\gamma)_{\gamma\in\Gamma},(\theta_\gamma)_{\gamma\in\Gamma})$ is a partial action of a discrete group $\Gamma$ on a totally disconnected space $X$, such that all the sets $X_\gamma$ are compact open in $X$, we can form the transformation groupoid $\Gamma\ltimes X$ given as
	$$\Gamma\ltimes X=\lbrace (\gamma,x)\in \Gamma\times X\mid x\in X_\gamma\rbrace,$$
	with operations given by
	$(\gamma,x)(\gamma',\theta_{\gamma^{-1}}(x))=(\gamma\gamma',x)$ and $(\gamma,x)^{-1}=(\gamma^{-1},\theta_{\gamma^{-1}}(x))$.
	This groupoid is always ample.
	If we denote the collection of compact open subsets of $X$ by $\mathbb{K}$, then one can define the type semigroup $S(X,\Gamma,\mathbb{K})$ (see \cite[Definition 7.1]{MR3144248} for the case of partial actions and \cite{RS12} for the special case of global actions). Then it is not hard to see that
	$$S(\Gamma\ltimes X,(\Gamma\ltimes X)^a)=S(X,\Gamma,\mathbb{K}).$$
	Indeed, the underlying sets of the two semigroups are the same, and one easily verifies, that our equivalence relation $\sim$ coincides with the equivalence relation $\sim_S$ defined in \cite{MR3144248}: Let $A=\bigcup_{i=1}^nA_i \times \lbrace i\rbrace$ and $B=\bigcup_{j=1}^m B_j \times\lbrace j\rbrace$ with $A_i,B_j\subseteq X$ compact open for all $1\leq i\leq n$ and $1\leq j\leq m$. If $A\sim_S B$ via the elements $\gamma_1,\ldots,\gamma_l$ and corresponding clopen subsets $C_1,\ldots, C_l$ with $C_k\subseteq X_{\gamma_k^{-1}}$, then $W_k:=\lbrace \gamma_k\rbrace \times \theta_{\gamma_k}(C_k)$ is a compact open bisection for each $1\leq k\leq l$ implementing the equivalence $A\sim B$.
	The converse is shown similarly. One just has to note that every compact open bisection $W$ of $\Gamma\ltimes X$ is a finite (disjoint) union of bisections $$W=\bigsqcup_i\lbrace \gamma_i\rbrace\times C_i$$ with $C_i\subseteq X_\gamma$ a compact open subset.
	Since $W$ is a bisection it follows that also the $r(\lbrace\gamma_i\rbrace \times C_i)$ and $d(\lbrace\gamma_i\rbrace \times C_i)$ are pairwise disjoint respectively. If we apply this fact to each compact open bisection appearing in an implementation of the equivalence $A\sim B$, we obtain the group elements and compact open subsets of $X$ implementing an equivalence $A\sim_S B$.
\end{rem}

\begin{rem}\label{Remark:SetInclusionImpliesInequalityInTypeSemigroup}
Consider compact open subsets $A\subseteq B\subseteq G^{(0)}$. Let $C=B\setminus A\subseteq G^{(0)}$. Then $C\in G^a$ and one immediately checks that $A\times\lbrace 1\rbrace \cup C\times \lbrace 2\rbrace \sim B\times \lbrace 1\rbrace$. Hence $[A]\leq [B]$ in $S(G,G^a)$.
\end{rem}
The next lemma shows, how the type semigroup reflects paradoxicality:
\begin{lem}\label{Lem:ParadoxicalityInTermsOfTypeSemigrp}
	Let $G$ be an ample groupoid and $k>l>0$. If $A\subseteq G^{(0)}$ is compact open, then
	$A$ is $(G^a,k,l)$-paradoxical if and only if $k[A]\leq l[A]$ in $S(G,G^a)$.
	\end{lem}
	\begin{proof}
		Suppose $k[A]\leq l[A]$. Then there exists $B=\bigcup_{j=1}^m B_j\times \lbrace j\rbrace$ such that $k[A]+[B]=l[A]$. Thus we have
		$$\bigcup\limits_{i=1}^k A\times\lbrace i\rbrace \cup \bigcup\limits_{j=1}^m B_j\times \lbrace k+j\rbrace \sim \bigcup\limits_{i=1}^l A\times \lbrace i\rbrace$$
		It follows, that there exists $r\in\N$, $W_1,\ldots, W_r\in G^a$ and $n_1,\ldots, n_r\in\lbrace 1,\ldots, k+m\rbrace$, $m_1,\ldots, m_r\in \lbrace 1,\ldots, l\rbrace$ such that
		$$\bigcup\limits_{i=1}^l A\times\lbrace i\rbrace=\bigsqcup\limits_{i=1}^r r(W_i)\times \lbrace m_i\rbrace$$
		and $$\bigcup\limits_{i=1}^k A\times \lbrace i\rbrace\cup \bigcup\limits_{j=1}^m B_j\times \lbrace k+j\rbrace=\bigsqcup\limits_{i=1}^r d(W_i)\times \lbrace n_i\rbrace$$
		The first equality entails that the sets $r(W_i)\times \lbrace m_i\rbrace\subseteq A\times \lbrace 1,\ldots, l\rbrace$ are pairwise disjoint. From the second equality we see that we can write
		$$A=\bigcup\limits_{\lbrace i \mid n_i=1\rbrace} d(W_i)=\ldots =\bigcup\limits_{\lbrace i \mid n_i=k\rbrace}d(W_i).$$
		Thus, by renumbering the $W_i$ appropriately, we obtain a $(G^a,k,l)$-paradoxical decomposition of $A$.

		Assume conversely, that $A$ is $(G^a,k,l)$-paradoxical. Then there exist numbers $n_i$ for each $1\leq i\leq k$, bisections $V_{i,j}\in G^a$ and numbers $m_{i,j}\in \lbrace 1,\ldots l\rbrace$ such that $A=\bigsqcup_{j=1}^{n_i}d(V_{i,j})$ and $r(V_{i,j})\times\lbrace m_{i,j}\rbrace\subseteq A\times \lbrace 1,\ldots, l\rbrace$ pairwise disjoint.
		We then have
		$$\bigcup\limits_{i=1}^k A\times\lbrace i\rbrace = \bigcup\limits_{i=1}^k \bigcup\limits_{j=1}^{n_i}d(V_{i,j})\times\lbrace i\rbrace\sim \bigcup\limits_{i=1}^k \bigcup\limits_{j=1}^{n_i} r(V_{i,j})\times\lbrace m_{i,j}\rbrace
		= \bigcup\limits_{p=1}^l A_p\times \lbrace p\rbrace, $$
		where $A_p=\bigsqcup_{\lbrace (i,j)\mid m_{i,j}=p\rbrace}r(V_{i,j})\subseteq A$.
		Applying Remark \ref{Remark:SetInclusionImpliesInequalityInTypeSemigroup} we conclude $$k[A]=[\bigcup\limits_{i=1}^k A\times\lbrace i\rbrace]=[\bigcup\limits_{p=1}^l A_p\times\lbrace p\rbrace]=\sum\limits_{p=1}^l [A_p] \leq l[A].$$
	\end{proof}
	
	\begin{rem} In Remark \ref{Remark:paradoxical} (1) we have seen that if a compact open subset $A\subseteq G^{(0)}$ is $(G^a,k,l)$-paradoxical, then it is $(G^a,k',l')$-paradoxical for all $k\geq k'>l'\geq l>0$.  
		From the characterization of a paradoxical decomposition in terms of $S(G,G^a)$ it is easy to see that $k'$ may in fact be chosen bigger than $k$, i.e. $A$ admits a $(G^a,k',l')$-paradoxical decomposition for all $k'>l'\geq l>0$.
	\end{rem}

\begin{prop}\label{Prop:K-theoryTypeSemigroup}
	Let $G$ be an ample groupoid with $G^{(0)}$ compact. Then there exists a faithful, order preserving, surjective monoid homomorphism
	$$\rho:C(G^{(0)},\Z)^+\rightarrow S(G,G^a)$$
	such that for all $S\in G^a$ and $f\in C(G^{(0)},\Z)^+$ with $supp(f)\subseteq r(S)$ we have $\rho(f)=\rho(f\circ \alpha_S)$, where $\alpha_S:d(S)\rightarrow r(S)$ is the canonical homeomorphism associated to the bisection $S\in G^a$.
\end{prop}
\begin{proof}
	Given $f\in C(G^{(0)},\Z)^+$ there exist clopen subsets $A_i\subseteq G^{(0)}$ such that $f=\sum_{i=1}^n 1_{A_i}$. Of course such a representation will not be unique in general but if $f= \sum_{i=1}^n 1_{A_i}=\sum_{i=1}^m 1_{B_i}$, then $$\bigcup\limits_{i=1}^n A_i\times \lbrace i\rbrace\sim \bigcup\limits_{i=1}^m B_i\times \lbrace i\rbrace.$$
	Indeed, let $(C_j)_{j=1}^l$ be a common clopen refinement of the families $(A_i)_i$ and $(B_i)_i$, i.e. each $A_i$ and each $B_i$ is a union of pairwise disjoint $C_j$. If we put $n_j=\abs{\lbrace i\mid C_j\subseteq A_i\rbrace}=\abs{\lbrace i\mid C_j\subseteq B_i\rbrace}$, then $f=\sum_{j=1}^l n_j 1_{C_j}$. Let $D_{j,i}:= C_j$ if $C_j\subseteq A_i$ and the empty set otherwise. We get $\bigsqcup_{j=1}^l D_{j,i}=A_i$ for every $1\leq i\leq n$. Thus
	$$ \bigcup\limits_{i=1}^n A_i\times\lbrace i\rbrace
	=\bigcup\limits_{i=1}^n\bigsqcup\limits_{j=1}^l D_{j,i}\times \lbrace i\rbrace
	=\bigsqcup\limits_{j=1}^l\bigsqcup\limits_{\lbrace i\mid C_j\subseteq A_i\rbrace} C_j\times \lbrace i\rbrace
	\sim \bigsqcup\limits_{j=1}^l\bigsqcup\limits_{i=1}^{n_j} C_j\times \lbrace i\rbrace=:C $$
	The same argument shows that $\bigcup_{i=1}^m B_i\times \lbrace i\rbrace\sim C$ and hence our claim follows by transitivity.
	Consequently for $f=\sum_{i=1}^n 1_{A_i}$ it makes sense to define $$\rho(f):=\left[\bigcup\limits_{i=1}^n A_i\times\lbrace i\rbrace\right].$$
	It is straightforward to check, that $\rho$ is a surjective monoid homomorphism.
	To see the faithfulness, it is enough to note that if
	$\bigcup_{i=1}^n A_i\times \lbrace i\rbrace\sim \emptyset\times \lbrace 1\rbrace$, then also $A_i=\emptyset$ for all $1\leq i\leq n$.
	Finally, let $S\in G^a$ and $f\in C(G^{(0)},\Z)^+$ with $supp(f)\subseteq r(S)$. Then we can write $f=\sum_{i=1}^n 1_{A_i}$ for compact open subsets $A_i\subseteq r(S)$. Let $S_i:=r_{\mid S}^{-1}(A_i)\subseteq S$. Then all the $S_i\in G^a$ with $r(S_i)=A_i$ by construction. It is not hard to see that $f\circ \alpha_S=\sum_{i=1}^n 1_{d(S_i)}$.
	On the other hand it is obvious from the construction that $\bigcup_{i=1}^n A_i\times\lbrace i\rbrace\sim \bigcup_{i=1}^n d(S_i)\times\lbrace i\rbrace$. Hence we have established $\rho(f)=\rho(f\circ \alpha_S)$.
\end{proof}
A state on a preordered abelian monoid $S$ is an additive, order preserving map $\varphi:S\rightarrow [0,\infty]$ with $f(0)=0$. We say that $\varphi$ is \textit{trivial} if $\varphi(x)\in\lbrace 0,\infty\rbrace$ for all $x\in S$ and \textit{non-trivial} otherwise.

We will first see how in the minimal setting non-trivial states on the monoid $S(G,G^a)$ give rise to traces on $C_r^*(G)$. In order to do this we need some more preparation:\\
Recall that an element $u$ in an abelian monoid $S$ is called an \textit{order unit} if $u\neq 0$ and for every element $x\in S$ there exists an $n\in\N$ such that $x\leq nu$. The monoid $S$ is called \textit{simple} provided that every non-zero element of $S$ is an order unit. If $S$ is a simple monoid, then every non-trivial state $\varphi:S\rightarrow [0,\infty]$  is faithful and $\varphi(x)<\infty$ for all $x\in S$: Suppose there exists a $0\neq u\in S$ with $0<\varphi(u)<\infty$. Let $x\in S$ be any non-zero element. Then $u$ and $x$ must be order units by simplicity of $S$. Hence there exist $n,m\in\N$ such that $x\leq nu$ and $u\leq mx$. Applying the state gives us $\varphi(x)\leq n\varphi(u)<\infty$ and $0<\varphi(u)\leq m\varphi(x)$ and thus $0<\varphi(x)<\infty$. As $0\neq x$ was arbitrary the claim follows.\\
It is not surprising that minimality of the groupoid is sufficient to ensure simplicity of the type semigroup, which is the content of the following lemma.
\begin{lem}\label{Lemma:Simple Type Semigroup}
	Let $G$ be an ample groupoid. If $G$ is minimal, then $S(G,G^a)$ is simple.
\end{lem}
\begin{proof} Fix $0\neq x\in S(G,G^a)$.
	We have to show, that $x$ is an order unit. Let us first assume that $x=[A]$, where $A\subseteq G^{(0)}$ is a single compact open subset.
	Consider $y=[B]$ for another compact open subset $B$ of $G^{(0)}$. Since $G$ is minimal and $A\neq \emptyset$ we have $GA=G^{(0)}$. In particular $B\subseteq GA$. As $B$ is compact there exist finitely many bisections $S_1,\ldots S_n\in G^a$ such that
	$$B\subseteq \bigcup\limits_{i=1}^n d(S_i),\text{ and } r(S_i)\subseteq A.$$
	By passing to a refinement if necessary we may assume that the sets $d(S_i)$ are pairwise disjoint. But then
	$$B\times\lbrace 1\rbrace \subseteq \bigsqcup\limits_{i=1}^n d(S_i)\times\lbrace 1\rbrace \sim \bigsqcup\limits_{i=1}^n r(S_i)\times\lbrace i\rbrace \subseteq \bigcup\limits_{i=1}^n A\times \lbrace i\rbrace,$$
	and hence $[B]\leq n[A]$. Now if $y\in S(G,G^a)$ is arbitrary, then $y=[\bigcup_{j=1}^m B_j\times\lbrace j\rbrace]$. By the above argument for each $1\leq j\leq m$ there exist $n_j\in\N$ such that $[B_j]\leq n_j[A]$ and hence we get $y=\sum_{j=1}^m [B_j]\leq (\sum_{j=1}^m n_j)[A]$. Thus we have shown that $x=[A]$ is an order unit.\\
	Finally suppose that $0\neq x=[\bigcup_{i=1}^n A_i\times\lbrace i\rbrace]$ is arbitrary and $y\in S(G,G^a)$. Then for each $1\leq i\leq n$ we can find a number $k_i\in\N$ such that $y\leq k_i[A_i]$. Let $k:=\max\lbrace k_1,\ldots k_n\rbrace$. Then we have $y\leq \sum_{i=1}^nk_i[A_i]\leq k\sum_{i=1}^n [A_i]=kx$ which finishes the proof.
\end{proof}

We are now ready to see how states on $S(G,G^a)$ give rise to traces on the reduced groupoid $C^*$-algebra.
\begin{lem}\label{Lemma:faithful states on type sgrp lift to traces}
Let $G$ be an ample minimal groupoid with compact unit space. Suppose $\varphi:S(G,G^a)\rightarrow [0,\infty)$ is a faithful state with $\varphi([G^{(0)}])=1$. Then $\varphi$ lifts to a faithful tracial state $\tau:C_r^*(G)\rightarrow \C$ with $\tau(1_C)=\varphi([C])$ for all clopen subsets $C\subseteq G^{(0)}$.
\end{lem}
\begin{proof}
Consider the composition $$\tilde{\tau}:=\varphi\circ \rho:C(G^{(0)},\Z)^+\rightarrow [0,\infty),$$
where $\rho:C(G^{(0)},\Z)^+\rightarrow S(G,G^a)$ is the invariant, faithful surjective monoid homomorphism provided by Proposition \ref{Prop:K-theoryTypeSemigroup}. Note that $\tilde{\tau}$ is an additive, order preserving map since $\rho$ and $\varphi$ have these properties. Furthermore $\tilde{\tau}$ is faithful and inherits the $G$-invariance property from $\rho$ meaning
that for all $S\in G^a$ and $f\in C(G^{(0)},\Z)^+$ with $supp(f)\subseteq r(S)$ we have $\tilde{\tau}(f)=\tilde{\tau}(f\circ \alpha_S)$, where $\alpha_S:d(S)\rightarrow r(S)$ is the canonical homeomorphism associated to the bisection $S\in G^a$. Since every element $f\in C(G^{(0)},\Z)$ can be written as a difference $f=f_1-f_2$ with $f_1,f_2\in C(G^{(0)},\Z)$ we can extend $\tilde{\tau}$ to a well-defined $G$-invariant homomorphism $C(G^{(0)})\rightarrow \R$ by $\tilde{\tau}(f)=\tilde{\tau}(f_1)-\tilde{\tau}(f_2)$. Indeed, if $f_1-f_2=f_1'-f_2'$, then $f_1+f_2'=f_1'+f_2$ and hence $\tilde{\tau}(f_1)+\tilde{\tau}(f_2')=\tilde{\tau}(f_1')+\tilde{\tau}(f_2)$ by the additivity of $\tilde{\tau}$. Hence $\tilde{\tau}(f_1)-\tilde{\tau}(f_2)=\tilde{\tau}(f_1')-\tilde{\tau}(f_2')$ as desired.\\
It is a well-known fact that $$(K_0(C(G^{(0)})),K_0(C(G^{(0)}))^+, [1]_0)\cong (C(G^{(0)},\Z),C(G^{(0)},\Z)^+,1_{G^{(0)}}).$$
Making this identification we have constructed a state $\tilde{\tau}$ on $K_0(C(G^{(0)}))$. Applying \cite[Theorem 3.3]{BR92} we conclude that $\tilde{\tau}$ lifts to a state $\tau_0:C(G^{(0)})\rightarrow\C$ such that $\tau_0(1_C)=\varphi([C])$ for all clopen subsets $C\subseteq G^{(0)}$. It is clear from the construction that $\tau_0(1_C\circ\alpha_S)=\tau_0(1_C)$ for all clopen subsets $C\subseteq G^{(0)}$ with $C\subseteq r(S)$. Since the linear span of all characteristic functions supported on clopen subsets of $G^{(0)}$ form a dense subspace of $C(G^{(0)})$ we obtain
$\tau_0(f\circ \alpha_S)=\tau_0(f)$ for all $f\in C(G^{(0)})$ with $supp(f)\subseteq r(S)$ for some $S\in G^a$ by continuity of $\tau_0$.
Finally let $\tau:=\tau_0\circ E$, where $E:C_r^*(G)\rightarrow C(G^{(0)})$ is the canonical faithful conditional expectation. It is clear, that $\tau$ is a faithful state and the trace property follows from the $G$-invariance of $\tau_0$ (see \cite[Lemma 4.2]{LR17} for a recent proof of this fact).
\end{proof}

Recall that a preordered monoid $S$ is called \textit{almost unperforated} if whenever $x,y\in S$ and $n\in\N$ satisfy $(n+1)x\leq ny$, then $x\leq y$.

\begin{thm}\label{purely inf}
	Let $G$ be an ample groupoid with compact unit space which is minimal and topologically principal. Consider the following assertions:
	\begin{enumerate}
		\item Every clopen subset of $G^{(0)}$ is $(G^a,2,1)$-paradoxical.
		\item $C_r^*(G)$ is (strongly) purely infinite.
		\item $C_r^*(G)$ is traceless (i.e., every lower-semicontinuous 2-quasitrace is trivial).
		\item $C_r^*(G)$ is not stably finite.
		\item Every state on the type semigroup $S(G,G^a)$ is trivial.
	\end{enumerate}
	Then we have $(1)\Rightarrow(2)\Rightarrow(3)\Rightarrow(4)\Rightarrow (5)$.
	If additionally $S(G,G^a)$ is almost unperforated, then $(5)\Rightarrow (1)$ making all of these properties equivalent.
\end{thm}
\begin{proof}
	The implication $(1)\Rightarrow (2)$ is a special case of Corollary~\ref{Cor: Every set is paradoxical implies purely infinite}. The implication $(2)\Rightarrow (3)$ holds for all $C^*$-algebras (see \cite[Proposition~4.1 (ii) and Proposition~4.11]{KR04}). For $(3)\Rightarrow (4)$ follows from \cite[Remark~2.27 (viii)]{KR04}, which states that a simple $C^*$-algebra $B$ is stably finite if and only if there exists a faithful semi-finite lower-semicontinuous 2-quasitrace on $B$. $(4)\Rightarrow (5)$: Suppose there exists a non-trivial state on $S(G,G^a)$. Then upon normalizing we can assume that it is faithful, takes only finite values and sends $[G^{(0)}]$ to $1$. Lemma \ref{Lemma:faithful states on type sgrp lift to traces} implies that $C_r^*(G)$ admits a faithful tracial state. Hence it must be stably finite. A contradiction.
	
	Finally assume that $S(G,G^a)$ is almost unperforated. Now in the situation of $(5)$ for every clopen subset $A$ of $G^{(0)}$ there must exist a number $n\in \N$ such that $(n+1)[A]\leq n[A]$ by Theorem \ref{Thm:ExistenceOfStatesOnTypeSemigroup}. But then also $2(n+1)[A]\leq n[A]$ and hence $2[A]\leq [A]$. Thus $A$ is $(G^a,2,1)$-paradoxical by Lemma \ref{Lem:ParadoxicalityInTermsOfTypeSemigrp}.
\end{proof}
\begin{rem}
It is still unclear to which extent the type semigroup is almost unperforated. It was noted in \cite[Remark 4.7]{MR3158244} that the implication $(5)\Rightarrow (1)$ also holds under the weaker comparability assumption that there exist integers $n,m\geq 2$ such that for all $x,y\in S(G,G^a)$, $nx\leq my$ implies $x\leq y$.\\
It has recently been shown by Ara and Exel in \cite{MR3144248} that there exists an action of a finitely generated free group $\mathbb{F}$ on a totally disconnected second countable Hausdorff compact space $X$ such that the type semigroup $S(X,\mathbb{F}, \mathbb{K})=S(\mathbb{F}\ltimes X,(\mathbb{F}\ltimes X)^a)$ is \emph{not} almost unperforated, where $\mathbb{K}$ denotes the algebra of compact open subsets of $X$.

\end{rem}
We obtain a dichotomy result from Corollary~\ref{simple} and Theorem~\ref{purely inf}:
\begin{cor}
Let $G$ be an ample groupoid with compact unit space which is minimal and topologically principal. If $S(G,G^a)$ is almost unperforated, then $C_r^*(G)$ is a simple $C^*$-algebra which is either stably finite or strongly purely infinite.
\end{cor}
We may compare this corollary with a similar dichotomy result for separable $C^*$-algebras with finite nuclear dimension in \cite[Theorem~5.4]{WZ10}. However, we do not require the nuclearity of the groupoid $C^*$-algebra in the corollary. It is natural to ask for examples which satisfy all assumptions of the corollary with \emph{infinite} nuclear dimension.

Finally we can apply our findings to characterize stable finiteness of $C_r^*(G)$:
Recall that a unital $C^*$-algebra $A$ is called \textit{finite} if its unit $1\in A$ is a finite projection. It is called \textit{stably finite} if $M_n(A)$ is finite for all $n\in\N$.

\begin{thm}\label{stably finite} Let $G$ be an ample groupoid. Suppose that $G^{(0)}$ is compact and $G$ is minimal. Then the following are equivalent:
	\begin{enumerate}
		\item $C_r^*(G)$ admits a faithful tracial state;
		\item $C_r^*(G)$ is stably finite;
		\item Every clopen subset of $G^{(0)}$ is completely $G^a$-non-paradoxical;
		\item There exists a faithful state $\varphi:S(G,G^a)\rightarrow [0,\infty)$ with $\varphi([G^{(0)}])=1$.
	\end{enumerate}
\end{thm}
\begin{proof}
	$(1)\Rightarrow (2)$: It is well-known.\\
	$(2)\Rightarrow (3)$: This follows from Proposition \ref{Prop:Stably finite implies non-paradoxical}.\\
	$(3)\Rightarrow (4)$: It follows from Lemma \ref{Lem:ParadoxicalityInTermsOfTypeSemigrp} in combination with Theorem \ref{Thm:ExistenceOfStatesOnTypeSemigroup}, that $S(G,G^a)$ admits a non-trivial state. By Lemma \ref{Lemma:Simple Type Semigroup} and the discussion above that Lemma every such state is faithful and will only take finite values. Upon normalizing $\varphi$ we can assume that $\varphi([G^{(0)}])=1$.\\
	$(4)\Rightarrow (1)$: This is the content of Lemma \ref{Lemma:faithful states on type sgrp lift to traces}.
\end{proof}

This paper follows the general philosophy that under suitable freeness conditions on $G$ much of the structure of $C_r^*(G)$ can be seen at the level of the canonical abelian subalgebra $C_0(G^{(0)})$. And indeed our Theorems \ref{ThmA} and \ref{ThmB} seem to validate this point of view. We believe that many of our results presented in this work generalize to the case of twisted groupoid $C^*$-algebras and hence - in view of Renault's reconstruction theorem (see \cite{MR2460017}) - might be used to give abstract conditions on Cartan pairs ensuring a complete understanding of structural properties of the ambient $C^*$-algebra in terms of its Cartan subalgebra.

\ \newline
{\bf Acknowledgments}. A few days after we posted our paper on the arXiv, Rainone and Sims informed us, that they independently obtained results very similar to those in sections 4 and 5 of the present paper. We would like to thank them for citing our paper. 

In fact, their definition of paradoxical decomposition is more flexible than our definition. However, these two notions of paradoxical decomposition are equivalent by a Bernstein-Schr\"{o}der-type argument (see e.g. the proof of \cite[Proposition~3.2]{RS17}). This can also be deduced from the fact that their definition of the type semigroup agrees with our definition (see \cite[Remark~5.5]{RS17}).

\bibliographystyle{plain}

\end{document}